\documentclass[12pt]{amsart}
\usepackage{graphicx} 
\usepackage{amsmath,amssymb,amsthm, mathtools}
\usepackage[usenames,dvipsnames]{color}
\usepackage{enumitem}

\newcommand{\Kbar}{\overline{K}}
\newcommand{\PP}{\mathbb{P}}
\newcommand{\PKbar}{\mathbb{P}^1(\overline{K})}

\newcommand\sgn{\text{sgn}}

\newcommand\Z{\mathbb{Z}}
\newcommand\Q{\mathbb{Q}}

\newcommand\PGL{\text{PGL}}

\newenvironment{Gal}{{\rm Gal\,}}{}

\newenvironment{Aut}{{\rm Aut}}{}

\newtheorem{theorem}{Theorem}[section]
\newtheorem{lemma}[theorem]{Lemma}
\newtheorem{proposition}[theorem]{Proposition}
\newtheorem{proposition-definition}[theorem]{Proposition-Definition}
\newtheorem{corollary}[theorem]{Corollary}

\theoremstyle{definition}
\newtheorem{example}[theorem]{Example}
 
\newtheorem{definition}[theorem]{Definition}

\theoremstyle{remark}
\newtheorem*{remark}{Remark}

\numberwithin{equation}{section}

\begin{document}
\title{On the arithmetic of bicritical rational functions}
\author{Vefa Goksel} 
\author{Rafe Jones}

\begin{abstract}
Bicritical rational functions -- those with precisely two critical points -- include the well-studied families of unicritical polynomials and quadratic rational functions. In this article we lay out general foundations for studying arithmetic dynamical properties of bicritical rational functions, and prove new Galois-theoretic results for a family with special properties. We study the field of definition of the critical points, and give a normal form up to M\"obius conjugacy over this field. As a corollary, we show that after a finite extension of the ground field, the arboreal Galois representation attached to a bicritical rational function injects into an iterated wreath product of cyclic groups. We then examine the family of quadratic  $\phi \in \Q(x)$ with critical points $\gamma_1$ and $\gamma_2$ such that $\phi(\gamma_1) = \gamma_2$. Adapting methods of Odoni-Stoll in the polynomial case to rational functions, we show that the arboreal representation is surjective for an infinite subfamily.  
\end{abstract}

 \dedicatory{Dedicated to Nigel Boston, mentor and friend to us both} 
\date{\today}

\maketitle

\section{Introduction}

Let $K$ be a field of characteristic $\ell \geq 0$ with fixed algebraic closure $\overline{K}$, and let $\phi \in K(z)$ be a rational function of degree $d \geq 2$. We call $\phi$ \textit{bicritical} if it has precisely two critical points, i.e. there are distinct $\gamma_1, \gamma_2 \in \mathbb{P}^1(K)$ with ramification index $e_\phi(\gamma_i) > 1$, and $e_\phi(z) = 1$ for all other $z \in \mathbb{P}^1(K)$. We give full definitions in Section \ref{sec: background}. A bicritical rational function for which one of the critical points is also a fixed point is a \textit{unicritical polynomial}, since a linear fractional transformation (defined over $\overline{K}$) taking the fixed critical point to $\infty$ and the other critical point to $0$ conjugates the map to $z^d + c$ for some $c \in \overline{K}$. A rational function $\phi \in K(z)$ of degree 2 is bicritical provided that $\ell \neq 2$. 


From an arithmetic dynamical perspective, both unicritical polynomials and quadratic rational functions have received significant attention recently. For a small sample of recent work, see \cite{benedettounicrit}, \cite{BJKL}, \cite{tuckerunicrit}, \cite{habeggerunicrit}, \cite{han-tucker}, \cite{arbfew}. General bicritical rational maps over $\mathbb{C}$ have been studied in the context of complex dynamics, including by Milnor \cite{milnorbicrit} and others (e.g. \cite{epstein2012classification, KochBicriticalCommon, koch2023deckgroupsiteratesbicritical, pilgrimbicrit}).

In \cite{milnorbicrit}, Milnor showed that an arbitrary bicritical rational map in $\mathbb{C}(z)$  is conjugate over $\mathbb{C}$ to a map of the form $(c_1z^d + a)/(c_2z^d + b),$ obtained by conjugating the critical points to $0$ and $\infty$ via a linear fractional transformation. Our first main result gives a normal form of similar flavor, but suited to arithmetic applications. In particular, we give information on the field of definition of the conjugacy.

\begin{theorem} \label{intro: normalform}
Let $K$ be a field of characteristic $\ell \geq 0$, and let $\phi \in K(z)$ be bicritical of degree $d \geq 2$ with critical points $\gamma_1, \gamma_2 \in \PKbar$. Assume that $\ell = 0$ or $\ell > d$. If $\phi(\{\gamma_1, \gamma_2\}) \neq \{\gamma_1, \gamma_2\}$, then $\phi$ is conjugate over $K(\gamma_1, \gamma_2)$ to
\begin{equation} \label{intro: mainform} 
\frac{z^d + a}{z^d + b}
\end{equation}
for some $a, b \in K(\gamma_1, \gamma_2)$ with $a \neq b$.
\end{theorem}
This normal form is close to unique: a map of the form in Eq. \ref{intro: mainform} is conjugate to at most one other such map, besides itself. See Theorem \ref{normalform} for a more detailed theorem statement, and for precise statements on uniqueness.

Under the hypotheses of Theorem \ref{intro: normalform}, we show further that $[K(\gamma_1, \gamma_2) : K] \leq 2$, with equality holding if and only if $\gamma_1$ and $\gamma_2$ are Galois-conjugate over $K$ (Theorem \ref{critfield}). 
Thus we can only have $[K(\gamma_1, \gamma_2) : K] = 2$ if the orbits of $\gamma_1$ and $\gamma_2$ under $\phi$ are indistinguishable in a combinatorial sense; see Corollary \ref{critspec}. 

The arboreal Galois representations associated to bicritical rational functions are also of interest. Let $\phi \in K(z)$ have degree $d \geq 2$, let $\alpha \in K$ and for $n \geq 1$ take $\phi^{-n}(\alpha) = \{\beta \in \overline{K} : \phi^n(\beta) = \alpha\}$. Assume that $\phi^{-n}(\alpha)$ has $d^n$ distinct elements, for each $n \geq 1$. Let $K_n(\alpha) = K(\phi^{-n}(\alpha))$, $K_\infty(\alpha) = \bigcup_{n \geq 1} K_n(\alpha)$, and let $G_n(\alpha)$ and $G_\infty(\alpha)$ be the Galois groups over $K$ of $K_n(\alpha)$ and $K_\infty(\alpha)$, respectively. If we take $T^d_\infty =  \bigsqcup_{n \geq 0} \phi^{-n}(\alpha)$ and assign edges according to the action of $\phi$, then $T^d_\infty$ has the structure of a complete infinite $d$-ary rooted tree with root $\alpha$ (we take $\phi^0$ to be the identity map), and $G_\infty(\alpha)$ embeds naturally into $\text{Aut}(T^d_\infty)$.

In Corollary \ref{wreath}, we show that if $\phi$ is a bicritical rational function satisfying the hypotheses of Theorem \ref{intro: normalform}, and $K(\gamma_1, \gamma_2) = K$, then $G_n(\alpha)/G_{n-1}(\alpha)$ embeds into $(C_d)^{d^{n-1}}$ for each $n \geq 2$, where $C_d$ denotes the cyclic group of order $d$. In particular, if $K$ contains $\gamma_1, \gamma_2,$ and a primitive $d$th root of unity, then we obtain an injection
\begin{equation} \label{intro: wreathprod eqn}
    G_\infty(\alpha) \hookrightarrow [C_d]^\infty,
\end{equation}
where $[C_d]^\infty$ denotes the infinite iterated wreath product of $C_d$, and can be identified with a subgroup of $\text{Aut}(T^d_\infty)$. Note that $[C_2]^\infty$ coincides with $\text{Aut}(T^2_\infty)$.



Our investigations in this article are motivated in part by the search for other families within the class of bicritical rational functions that might mimic the structure exhibited by unicritical polynomials.  One candidate 
is the family of maps with a \textit{trailing critical point}: $\phi^n(\gamma_1) = \phi^m(\gamma_2)$ for some $n > m \geq 0$. Note that this does not include the case of colliding critical points given by $\phi^n(\gamma_1) = \phi^n(\gamma_2)$, as studied in \cite{benedettocollision}. We remark that a bicritical rational function with a trailing critical point must have $K(\gamma_1, \gamma_2) = K$ by Corollary \ref{critspec}. 

Several authors have shown that the image of the map in Eq. \ref{intro: wreathprod eqn} has finite index in $[C_d]^\infty$ when $\phi$ is a unicritical polynomial satisfying various hypotheses (e.g.  \cite{tuckerunicrit}, \cite{han-tucker}, \cite{hindes2}, \cite{Li-stoll1}, \cite{Li-stoll2}, \cite{stoll}). In \cite[Theorem 5]{arbfew}, Juul et al. study a one-parameter family of quadratic maps with $\phi(\gamma_1) = \gamma_2$, and prove the map in Eq. \ref{intro: wreathprod eqn} is surjective for $\alpha = 0$. Using completely different methods, inspired by Stoll \cite{stoll} (who in turn drew on Odoni \cite[Section 4]{odoni}), we reach the same conclusion for a different family of quadratic maps with $\phi(\gamma_1) = \gamma_2$ and $\alpha = 0$. 


\begin{theorem} \label{intro: galois}
Let $m$ be a positive integer with $m \not\equiv 1 \pmod{4}$, and let 
$$\phi(z) = (z^2 + a)/z^2 \quad \text{with 
$a = -2(2m^2-1)^2.$}
$$ 
Then $[K_n(0) : K_{n-1}(0)] = 2^{2^{n-1}}$ for all $n \geq 1$, implying that $G_\infty(0) \cong \text{Aut}(T^2_\infty)$.
\end{theorem}

Observe that Theorem \ref{intro: normalform} and Corollary \ref{critspec} imply that if $\phi \in \Q(z)$ satisfies $\phi(\gamma_1) = \gamma_2$, and $\gamma_1$ and $\gamma_2$ have infinite forward orbit under $\phi$, then $\phi$ is conjugate over $\Q$ to $(z^2 + a)/z^2$ for some $a \in \Q$, so the form of $\phi$ in Theorem \ref{intro: galois} is general in some sense.
In Theorem \ref{thm:surjectivityB} we prove Theorem \ref{intro: galois}, and in fact we establish the conclusion for a somewhat larger class of $m$. We also obtain the weaker conclusion that $[K_n(0) : K_{n-1}(0)] = 2^{2^{n-1}}$ for all non-square-free $n$ provided only that $a \equiv 2 \pmod{4}$ and $a \leq -3$. See Theorem \ref{thm:surjectivityA}.

When put into the normal form of Theorem \ref{intro: galois}, the family studied in \cite{arbfew} is 
$$
\frac{z^2 + (1-b)/b^3}{z^2}, \quad b \in \Z
$$
with $\alpha = -1/b$. The proof makes fundamental use of the fact that $\alpha$ is strictly pre-periodic under $\phi$ (indeed, $\phi^2(\alpha) = \phi(\alpha) \neq \alpha$), a property it shares with the example studied in \cite[Theorem 1.2]{galrat}. By contrast, $\alpha = 0$ has infinite forward orbit for the family in our Theorem \ref{intro: galois}. Our method thus diverges significantly from that of \cite{arbfew}. We use the basic idea of taking advantage of $\alpha = 0$ being a critical point for our family, leading to special divisibility properties of its forward orbit. 

Stoll's paper \cite{stoll} established the surjectivity of the map in Eq. \ref{intro: wreathprod eqn} for $\phi(x) = x^2 + a$ and $\alpha = 0$, where $a$ belonged to the union of certain arithmetic progressions. The method is extended to further arithmetic progressions in \cite{Li-stoll1}, \cite{Li-stoll2}. 

In generalizing the methods of \cite{stoll} to the setting of non-polynomial rational functions, we encounter several technical obstacles. One of them is to reduce the theorem to showing that a certain recursively-defined sequence of positive integers $(|\theta_n|)_{n \geq 1}$ contains no squares. We do this in Lemma \ref{lemma:sufficient-theta_n}. The $\theta_n$ are constructed from the sequence $(p_n(0))_{n \geq 1}$, where $p_n(z)$ is the canonical choice for numerator of $\phi^n$ (see Section \ref{sec: defining}). It is crucial that $(p_n(0))_{n \geq 1}$ obey certain divisibility properties. 

We study $(p_n(0))_{n \geq 1}$ in some generality in Section \ref{sec: rigid}. A sequence $(c_n)_{n \geq 1}$ in a field $K$ with a complete set $A$ of non-archimedean absolute values is an \textit{$S$-rigid divisibility sequence} if $|c_n|<1$ implies $|c_n| = |c_{kn}|$ for all $k\geq 1$, and $|c_m|<1$ and $|c_n|<1$ imply $|c_{\gcd(m,n)}|<1$, for all absolute values outside of the finite set $S \subseteq A$. We generalize results in the polynomial case in \cite{quaddiv} and \cite{rice} to the setting of rational functions.
\begin{theorem}
\label{intro: rigid}
Let $K$ be the field of fractions of a Dedekind domain $R$, and fix a complete set $A$ of non-archimedean absolute values on $K$. Let $\phi \in K(z)$ have degree $d \geq 2$, and let $p,q \in R[z]$ be relatively prime polynomials with $\phi = p/q$.
Let $p_n$ and $q_n$ be as in Eq. \ref{eq:p_n,q_n}, and let $S \subseteq A$ be the set of absolute values at which either $\phi$ has bad reduction or the pair $p,q$ is not normalized. Assume that 
$$
p'(0) = q'(0) = 0.
$$
Then the sequence $(p_n(0))_{n\geq 1}$ is an $S$-rigid divisibility sequence.
\end{theorem}
See Section \ref{subsection: normalized and reduction} for background on normalized form. When $R$ is a PID the absolute values at which the pair $p,q$ is not normalized may be eliminated; See Proposition \ref{prop:PID}. The absolute values where $\phi$ has bad reduction, on the other hand, present deeper obstacles to rigid divisibility. See Example \ref{ex:rigid} for an illustration.

A second piece of the proof of Theorem \ref{intro: galois} is establishing the irreducibility over $\Q$ of the polynomials $p_n$. We make use of results in \cite{galrat} for this, as well as for calculating the discriminant of $p_n$.

The final major step in the proof is to show $|\theta_n|$ is not a square for $n \geq 1$. To accomplish this, we show that $|\theta_n|$ coincides modulo rational squares with 
\begin{equation} \label{eq:prod}
\prod_{d \mid n} (\phi^d(0))^{\mu(n/d)}.
\end{equation}
We then look for primes $p$ such that $\phi^i(0)$ maps to a fixed point modulo $p$ after a small number of iterates. This allows us to compute 
Eq. \ref{eq:prod} modulo $p$, and when $p$ is well-chosen the expression in Eq. \ref{eq:prod} turns out to be a non-quadratic residue modulo $p$. 

The fact that $\phi(0) = \infty$ means that Eq. \ref{eq:prod} is not well-defined when $n$ is square-free. This is a significant obstacle, which we circumvent by replacing $0$ with non-zero $\alpha$ such that $\phi^i(\alpha) = \phi^i(0)$ for some $i$. The smallest $i$ for which such $\alpha$ exists is $i = 3$, and happily the relevant curve has genus zero. To obtain $\alpha \in \Q$ we must take $a$ to be the $x$-coordinate of a rational point on this curve, which accounts for the hypothesis on $a$ in Theorem \ref{intro: galois}.

\section{Background on reduction and ramification for rational functions}

\label{sec: background}

We wish to develop a theory of bicritical rational functions in some generality, so in this section we give background material working over a general field. The standard reference is \cite[Chapters 1 and 2]{jhsdynam}, but we provide a self-contained treatment in order to accommodate our desired generality of the ground field.

\subsection{Defining polynomials for iterates of a rational function}

\label{sec: defining}

Let $K$ be a field with a fixed algebraic closure $\overline{K}$. A rational function $\phi$ of degree $d \geq 0$ defined over a field $K$ is a map $\PP^1 \to \PP^1$ given by $\phi([Z,W]) =  [P(Z,W), Q(Z,W)]$, where $P$ and $Q$ are degree-$d$ homogenous polynomials in $K[Z,W]$ with no common roots $[\alpha, \beta]$ in $\PP^1(\Kbar)$. For $n \geq 2$, the $n$th iterate of $\phi$ is given by $\phi^n = [P_n(Z,W), Q_n(Z,W)]$, where $P_1 = P$, $Q_1 = Q$, and for $n \geq 2$,
\begin{equation}
    \label{eq: PQdef}
\begin{split}
P_{n}(Z,W) & = P(P_{n-1}(Z,W), Q_{n-1}(Z,W))  \\
    Q_{n}(Z,W) & = Q(P_{n-1}(Z,W), Q_{n-1}(Z,W)).
\end{split}
\end{equation}

We have the following elementary but useful properties of $P_n$ and $Q_n$.
\begin{proposition} \label{prop:basic1}
    For all $n \geq 1$, $P_n$ and $Q_n$ are homogeneous polynomials of degree $d^n$ having no common roots in $\PP^1(\Kbar)$.
\end{proposition}

\begin{proof}
    To show the statement about common roots we induct on $n$; the base case $n = 1$ is true by assumption. If ${P_{n}}$ and ${Q_{n}}$ had a common root $[\alpha, \beta] \in \PP^1(\Kbar)$ for $n \geq 2$, then by Eq. \ref{eq: PQdef}, $[{P_{n-1}}(\alpha, \beta), {Q_{n-1}}(\alpha, \beta)]$ would be common root of ${P}$ and ${Q}$. Hence $[{P_{n-1}}(\alpha, \beta), {Q_{n-1}}(\alpha, \beta)]$ cannot be in $\PP^1(\overline{K})$, i.e. ${P_{n-1}}(\alpha, \beta) = {Q_{n-1}}(\alpha, \beta) = 0$. But this contradicts the inductive hypothesis. 

    A straightforward induction shows that both $P_n$ and $Q_n$ have the form $\sum_{i = 0}^{d^n} c_i Z^{d^n-i}W^{i}$ with $c_i \in K$. Because $P_n$ and $Q_n$ have no common roots in $\PP^1(\Kbar)$, neither can be identically zero, proving that both have degree $d^n$. 
\end{proof}

We frequently dehomogenize $P_n$ and $Q_n$ by taking $W = 1$, giving relatively prime $p_n(z), q_n(z) \in K[z]$ with $\max \{\deg p_n, \deg q_n\} = d^n$. Note that $\phi^n(z) = p_n(z)/q_n(z)$ provided $q_n(z) \neq 0$, and we interpret $\phi^n(z)$ as $\infty$ if $q_n(z) = 0$. Because $\phi^n = \phi \circ \phi^{n-1} = \phi^{n-1} \circ \phi$, we have $\phi^n = \phi(p_{n-1}/q_{n-1}) = \phi^{n-1}(p/q)$. 

Writing $p(z)=\sum_{i=0}^{d}a_iz^i$ and $q(z)=\sum_{i=0}^{e}b_iz^i$ and dehomogenizing Eq. \ref{eq: PQdef} gives 
$p_n(z) = P(p_{n-1}(z), q_{n-1}(z))$ and $q_n(z) = Q(p_{n-1}(z), q_{n-1}(z))$. More explicitly, if $p(z)=\sum_{i=0}^{d_p}a_iz^i$ and $q(z)=\sum_{i=0}^{d_q}b_iz^i$, with $\max\{d_p, d_q\} = d$, then for $n \geq 2,$
\begin{equation}
\label{eq:p_n,q_n}
\begin{split}
    p_{n} = \sum_{i=0}^{d_p} a_ip_{n-1}^iq_{n-1}^{d_p-i},\quad q_{n} = \sum_{i=0}^{d_q} b_ip_{n-1}^iq_{n-1}^{d_p-i} \qquad \text{if $d_p \geq d_q$;} \\
p_{n} = \sum_{i=0}^{d_p} a_ip_{n-1}^iq_{n-1}^{d_q-i},\quad q_{n} = \sum_{i=0}^{d_q} b_ip_{n-1}^iq_{n-1}^{d_q-i} \qquad \text{if $d_p < d_q$.}
\end{split}
\end{equation}

Observe that replacing $P$ and $Q$ by $cP$ and $cQ$ with $c \in K \setminus \{0\}$ does not change $\phi$. This creates ambiguity in speaking of ``the" numerator and denominator of $\phi$, and this ambiguity extends to iterates of $\phi$. We will disambiguate by fixing a choice of $P$ and $Q$ (or equivalenlty, $p$ and $q$) and noting that Eq.s \ref{eq: PQdef} and \ref{eq:p_n,q_n} specify $P_n$ and $Q_n$ (and thus $p_n$ and $q_n$) for all $n \geq 1$. 

\subsection{Normalized form and reduction with respect to a non-archimedean absolute value} \label{subsection: normalized and reduction}

Let $K$ be a field, $|\cdot|$ a non-archimedean absolute value on $K$, and $k$ the residue field of $K$ with respect to $|\cdot|$. We say that a pair $P,Q \in K[Z,W]$ of homogeneous polynomials is \textit{normalized} with respect to $|\cdot|$ if all coefficients of $P$ and $Q$ have absolute value at most one, and at least one coefficient of $P$ or $Q$ has absolute value 1. In this case we also say that their dehomogenizations $p,q \in K[z]$ are a normalized pair. If $\phi = [P,Q]$, the reduction $\tilde{\phi}$ of $\phi$ with respect to $|\cdot|$ is then defined as $[\tilde{P},\tilde{Q}]$, where the two entries are the coefficient-wise reductions of $P$ and $Q$.  Any other normalized pair $P'$ and $Q'$ with $\phi = [P',Q']$ yields the same reduction \cite[p.52]{jhsdynam}. We say $\phi$ has \textit{good reduction at $|\cdot|$} if $\deg \phi = \deg \tilde{\phi}$, or equivalently if $\phi = [P,Q]$ where the pair $P, Q$ is normalized and $\tilde{P}$ and $\tilde{Q}$ have no common roots in $\PP^1(\overline{k})$. See \cite[Section 2.3]{jhsdynam} for a complete treatment. 

\begin{proposition} \label{prop: reduction}
Fix a non-archimedean absolute value $|\cdot|$ on $K$, 
and let $\phi = [P,Q]$ be defined over $K$, where the pair $P, Q$ is normalized. Assume that $\phi$ has good reduction at $|\cdot|$. Then for all $n \geq 1$, the pair $P_n, Q_n$ is normalized and $\widetilde{P_n}$ and $\widetilde{Q_n}$ have no common roots in $\PP^1(\overline{k})$. In particular, $\max\{|p_n(0)|, |q_n(0)|\} = 1$.
\end{proposition}


\begin{proof} By hypothesis $\tilde{P}$ and $\tilde{Q}$ are homogeneous polynomials defined over $k$ with no common roots in $\PP^1(\overline{k})$. The reduction homomorphism respects the composition in Eq. \ref{eq: PQdef}, and we may thus apply Proposition \ref{prop:basic1} to conclude that $\widetilde{P_n}$ and $\widetilde{Q_n}$ have no common roots in $\PP^1(\overline{k})$. 

Because the pair $P,Q$ is normalized, they have no coefficients of absolute value exceeding 1. The same thus holds for $P_n$ and $Q_n$. If the pair $P_n, Q_n$ were not normalized, then both $\tilde{P_n}$ and $\tilde{Q_n}$ would be identically zero, contradicting the previous paragraph.
\end{proof}

\subsection{Ramification and tame Riemann-Hurwitz}

Let $f \in K[z]$ be a non-zero polynomial of degree $d \geq 0$. For any $\alpha \in \Kbar$ the polynomial 
$$g(z) = f(z + \alpha)  = \sum_{i = 0}^d b_iz^i \in \Kbar[z]
$$ 
has degree $d$ and satisfies $f(z) = g(z - \alpha)$. Define the order of vanishing $\text{ord}_\alpha(f)$ of $f$ at $\alpha$ to be the smallest $i \geq 0$ with $b_i \neq 0$.
Suppose that $\phi(z) = p(z)/q(z) \in K(z)$ is a non-zero rational function of degree $d \geq 1$, with $p(z)$ and $q(z)$ having no common root in $\Kbar$. Given $\alpha \in \Kbar$ we see that $p(z)q(\alpha) - q(z)p(\alpha)$ cannot be identically zero, and we define the ramification index of $\phi$ at $\alpha$ to be
\begin{equation} \label{ramdef}
e_\alpha(\phi) = \textrm{ord}_\alpha(p(z)q(\alpha) - q(z)p(\alpha)).
\end{equation}
Observe that $1 \leq e_\alpha(\phi) \leq d$. When $q(\alpha) \neq 0$, i.e. $\phi(\alpha) \neq \infty$, $e_\alpha(\phi)$ is the order of vanishing of the numerator of $\phi(z) - \phi(\alpha)$. Let $\psi(z) = 1/\phi(1/z)$, and define $e_\infty(\phi) = e_0(\psi)$. One easily checks that the ramification index is conjugation-invariant: if $\mu \in \PGL_2(\Kbar)$ is a linear fractional transformation and $\psi = \mu \circ \phi \circ \mu^{-1}$, then $e_\alpha(\phi) = e_{\mu(\alpha)}(\psi)$.

We say that $\alpha \in \PKbar$ is a \textit{critical point} for $\phi \in K(z)$ if $(\mu \circ \phi \circ \mu^{-1})'(\mu(\alpha)) \neq 0$, where $\mu \in \PGL_2(\Kbar)$ satisfies $\mu(\alpha) \neq \infty$ and $\mu(\phi(\alpha)) \neq \infty$. Note that $\alpha$ is a critical point if and only if $e_\alpha(\phi) > 1$. We have the following elementary result, whose proof we include for completeness.
\begin{proposition} \label{critical numerator}
Let $K$ be a field of characteristic $\ell \geq 0$ and $\phi(z) = p(z)/q(z) \in K(z)$ a rational function of degree $d \geq 2$, where $p(z)$ and $q(z)$ have no common root in $\Kbar$. If $\ell \nmid e_\alpha(\phi)$ for some $\alpha \in \PKbar$, then $p'(z)q(z) - q'(z)p(z)$ is not identically zero and
\begin{equation} \label{firstord}
e_\alpha(\phi) - 1 = \textrm{ord}_\alpha (p'(z)q(z) - q'(z)p(z))
\end{equation}
if $\alpha \in \Kbar$ and 
\begin{equation} \label{inford}
e_\alpha(\phi) - 1 = 2d - 2 - \deg[p'(z)q(z) - q'(z)p(z)]
\end{equation}
if $\alpha = \infty$.
\end{proposition}

\begin{proof}
Let $\alpha \in \Kbar$ and put $e = e_\alpha(\phi)$. From \eqref{ramdef} we obtain 
\begin{equation} \label{eq: order of p over q}
p(z + \alpha)q(\alpha) - q(z+\alpha)p(\alpha) = z^e h(z)
\end{equation}
for some $h \in \Kbar[z]$ with $h(0) \neq 0$. Suppose that $q(\alpha) \neq 0$. Dividing Eq. \ref{eq: order of p over q} through by $q(z + \alpha)q(\alpha)$, differentiating, and multiplying through by $(q(z + \alpha)q(\alpha))^2$ yields 
$$p'(z + \alpha)q(z+\alpha) - q'(z + \alpha)p(z+\alpha) = ez^{e-1}q(z + \alpha)q(\alpha)h(z) + z^eg(z),$$
for some $g \in \Kbar[z]$. 
Because $e \neq 0$ in $K$, $h(0) \neq 0$, and $q(\alpha) \neq 0$, we have that $p'(z)q(z) - q'(z)p(z)$ is not identically zero, and Eq. \ref{firstord} follows. 

If $q(\alpha) = 0$, then Eq. \ref{eq: order of p over q} becomes $q(z+\alpha)p(\alpha) = z^e h(z)$ for some $h \in \Kbar[z]$ with $h(0) \neq 0$, and thus $\textrm{ord}_\alpha(q(z)) = e$. Because $e \neq 0$ in $K$ and $p(\alpha) \neq 0$, it follows that $e - 1 = \textrm{ord}_\alpha(q'(z)) = \textrm{ord}_\alpha(q'(z)p(z))$, and in particular $q'(z)p(z)$ is not identically zero. But $\textrm{ord}_\alpha(p'(z)q(z)) \geq e$, and Eq. \ref{firstord} follows. 

In the case where $\alpha = \infty$, observe that $\deg (p'(z)q(z) - q'(z)p(z)) \leq 2d - 2$, and note that $1/\phi(1/z) = q_1(z)/p_1(z)$, where $q_1(z) = z^d q(1/z)$ and $p_1(z) = z^d p(1/z)$. By definition of $e_\infty(\phi)$ we then have
$$
q_1(z)p_1(0) - p_1(z)q_1(0) = z^e h(z)
$$
for some $h \in \Kbar[z]$ with $h(0) \neq 0$. Assume that $p_1(0) \neq 0$. Dividing by $p_1(z)p_1(0)$, differentiating, multiplying by $(p_1(z)p_1(0))^2$, and using $e \neq 0$ and $h(0) \neq 0$ shows that $\textrm{ord}_0(q_1'(z)p_1(z) - p_1'(z)q_1(z)) = e-1$, and in particular $q_1'(z)p_1(z) - p_1'(z)q_1(z)$ is not identically zero. One readily verifies that for any $f(z) \in \Kbar[z]$ and $n \geq \deg f$, $\textrm{ord}_0(z^nf(1/z)) = n - \deg f$ and $(z^nf(1/z))' = z^nf'(1/z)(-1/z^2) + nz^{n-1}f(1/z)$. 
A straightforward calculation then gives
$$
q_1'(z)p_1(z) - p_1'(z)q_1(z) = z^{2d}(q'(1/z)p(1/z) - p'(1/z)q(1/z))(-z^{-2}),
$$
from which Eq. \ref{inford} follows. 

Finally, the case where $\alpha = \infty$ and $p_1(0) = 0$ is handled similarly to the case where $\alpha \in \Kbar$ and $q(\alpha) = 0$ above. 
\end{proof}

\begin{corollary}[Tame Riemann-Hurwitz for $\mathbb{P}^1$] \label{cor: RH}
    Let $K$ be a field of characteristic $\ell \geq 0$ and $\phi \in K(z)$ a rational function of degree $d \geq 2$. If $\ell \nmid e_\alpha(\phi)$ for every $\alpha \in \PKbar$, then
    \begin{equation} \label{eq: RH}
        2d - 2 = \sum_{\alpha \in \PKbar} (e_\alpha(\phi) - 1).
    \end{equation}
\end{corollary}

\begin{proof}
    For non-zero $f \in K[z]$, unique factorization implies that $\sum_{\alpha \in \Kbar} \textrm{ord}_\alpha(f(z)) = \deg f$. The corollary then follows from Proposition \ref{critical numerator}.
\end{proof}

Under the hypotheses of Corollary \ref{cor: RH}, $\phi$ must have at least two critical points, since $e_\alpha(\phi) \leq d$. If $\phi$ is further assumed to be bicritical, then \eqref{eq: RH} gives $e_\phi(\gamma_1) = e_\phi(\gamma_2) = d$. We remark that if the assumption $\ell \nmid e_\alpha(\phi)$ is dropped, then it is possible for $\phi$ to have only a single critical point, though this requires $\ell \mid d$ or $\ell \mid d-1$. See \cite{faber} for more details. 


\section{Conjugacy and normal forms for bicritical rational functions}

We begin with a result on the smallest extension of $K$ containing the critical points of a bicritical rational function. By convention we take $K(\infty) = K$.

\begin{theorem} \label{critfield}
Let $K$ be a field of characteristic $\ell \geq 0$, and let $\phi \in K(z)$ be bicritical of degree $d \geq 2$ with critical points $\gamma_1, \gamma_2 \in \PKbar$. Assume that $\ell = 0$ or $\ell > d$. If $K(\gamma_1, \gamma_2) \neq K$ then $[K(\gamma_1, \gamma_2) : K] = 2$ and $\gamma_1$ and $\gamma_2$ are Galois-conjugate over $K$.
\end{theorem}

\begin{proof}   
    We begin by noting the standard fact that if $\ell \nmid n$, then any irreducible polynomial in $K[z]$ of degree $n$ must have distinct roots in $\overline{K}$.  
    
    Assume that $\infty \in \{\gamma_1, \gamma_2\}$, and without loss of generality say $\gamma_1 = \infty$. By Corollary \ref{cor: RH} we have $e_{\gamma_1}(\phi) = e_{\gamma_2}(\phi) = d-1$. Proposition \ref{critical numerator} then shows that $\gamma_2$ is the only root of the degree-$(d-1)$ polynomial $p'q - q'p \in K[z]$. But $\ell \nmid d-1$, and so any irreducible factor in $K[z]$ of $p'q - q'p$ can have only the single root $\gamma_2$, and hence has degree 1. It follows that $z- \gamma_2 \in K[z]$ and thus $K(\gamma_1, \gamma_2) = K$.

    We thus have reduced to the case $\{\gamma_1, \gamma_2\} \subset \Kbar$. Applying Corollary \ref{cor: RH} and Proposition \ref{critical numerator} shows that the degree-$(2d-2)$ polynomial $p'q - q'p$ has only the roots $\gamma_1$ and $\gamma_2$. Because $\ell \nmid 2d-2$, no irreducible factor in $K[z]$ of $p'q - q'p$ can have repeated roots. Hence either $p'q - q'p$ has two irreducible factors of degree one (with roots $\gamma_1$ and $\gamma_2$, respectively) or one irreducible factor of degree two (with roots $\gamma_1$ and $\gamma_2$). In the former case $K(\gamma_1, \gamma_2) = K$, while in the latter $[K(\gamma_1, \gamma_2) : K] = 2$, $K(\gamma_1, \gamma_2)$ is Galois over $K$, and $\gamma_1$ and $\gamma_2$ are Galois-conjugate over $K$.
\end{proof}

In the case where $d = 2$ we have the following description of maps for which $K(\gamma_1, \gamma_2) \neq K$, up to conjugacy by an element of $\PGL_2(K)$. For $\mu \in \PGL_2(\overline{K})$, we denote $\mu \circ \phi \circ \mu^{-1}$ by $\phi^\mu$.

\begin{proposition} \label{quadratic}
    Let $K$ be a field of characteristic $\neq 2$, and let $\phi \in K(z)$ be bicritical of degree $2$ with critical points $\gamma_1, \gamma_2 \in \PKbar$. Assume that $K$ has at least seven elements. If $K(\gamma_1, \gamma_2) \neq K$, then $\phi$ is conjugate over $K$ to a map of the form
    \begin{equation} \label{quadratic normal form}
        \frac{z^2 + az + r}{z^2 + bz + r},   
    \end{equation}
    where $a, b, r \in K$ and $K(\gamma_1, \gamma_2) = K(\sqrt{r})$. 
\end{proposition}

\begin{proof}
By Theorem \ref{critfield}, there is a non-square $s \in K$ and $c_0, c_1 \in K$ with $c_1 \neq 0$ and $\{\gamma_1, \gamma_2\} = \{c_0 \pm c_1 \sqrt{s}\}$. Conjugation of $\phi$ by $\mu(z) = (z-c_0)/c_1$ gives a map $\phi_1$ with $\{\gamma_1, \gamma_2\} = \{\pm \sqrt{s}\}$. 

We now argue that $\phi_1$ is conjugate over $K$ to a map $\phi_2$ whose critical points are $\pm \sqrt{s}$ and with $\phi_2(\infty) \not\in \{0, \infty\}$. One calculates that $\mu \in \PGL_2(K)$ preserves $\{\pm \sqrt{s}\}$ if and only if $\mu(z) = z$ or 
\begin{equation} \label{muform}
\mu(z) = \frac{c_2z \pm s}{z \pm c_2},
\end{equation}
where $c_2 \in K$ and the $\pm$ in the numerator and denominator have the same sign. If $\phi_1(\infty) \not\in \{0, \infty\}$, we take $\phi_2 = \phi_1$. Otherwise, conjugating $\phi_1$ by $\mu$ of the form \eqref{muform} (using the minus signs) gives $\phi_2$ with critical points $\pm \sqrt{s}$ and 
$$\phi_2(\infty) = \phi_1^\mu(\mu(c_2)) = \mu(\phi_1(c_2)).$$ 
Because $\deg \phi_1 = 2$, $\phi_1(c_2) = c_2$ has at most three solutions in $K$, and the same holds for $\phi_1(c_2) = s/c_2$ (where we take $s/c_2 = \infty$ if $c_2 = 0$). Because $\#K \geq 7$ we may take $c_2 \in K$ with $\mu(\phi_1(c_2)) \not\in \{0, \infty\}$.

Finally, conjugate $\phi_2$ by the scaling $z \mapsto c_3z$ where $c_3 = 1/\phi_2(\infty)$. We then have $\phi_3(\infty) = 1$ and the critical points of $\phi_3$ are $\pm c_3 \sqrt{s}$. Let $r = c_3^2s$ and note that for $\phi_3$ we have $K(\gamma_1, \gamma_2) = K(\sqrt{r})$. Therefore 
\begin{equation} \label{three}
\phi_3(z) = \frac{z^2 + a_1z + a_2}{z^2 + b_1z + b_2} \in K(z),
\end{equation}
and the numerator of $\phi_3'(z)$ is $c_4(z^2 - r)$ for some $c_4 \in K$. Differentiating \eqref{three}, setting the numerator equal to $c_4(z^2 - r)$ and equating coefficients yields $a_2 = b_2 = r$, giving the form in \eqref{quadratic normal form}.
\end{proof}

\begin{remark}
    When $K$ has fewer than seven elements one encounters examples such as 
    $$
    \phi(z) = \frac{z^2 - z + 2}{2z} \in \mathbb{F}_5(z).
    $$
    This map has critical points $\pm \sqrt{2}$ and fixes 
    $\infty$. Moreover, every $y \in \PP^1(\mathbb{F}_5)$ is a solution to either $\phi(y) = y$ or $ \phi(y) = 2/y$. Hence if $\mu$ is as in \eqref{muform}, we have that $\phi^{\mu}(\infty) \in \{0, \infty\}$ for every choice of $c_2$.
\end{remark}


A consequence of Theorem \ref{critfield} is that certain bicritical maps $\phi \in K(z)$ must have critical points defined over $K$. A \textit{critical orbit relation} for $\phi \in K(z)$ is any equality of the form $\phi^n(\gamma_i) = \phi^m(\gamma_j)$ where $\gamma_i$ and $\gamma_j$ are (not necessarily distinct) critical points for $\phi$ and $n, m \geq 0$. Recall that $\alpha \in \Kbar$ is pre-periodic under $\phi$ if there are $s > t \geq 0$ such that 
$\phi^s(\alpha) = \phi^t(\alpha)$. When $s$ and $t$ are both minimal, we call $t$ the pre-period of $\alpha$ and $s - t$ the period of $\alpha$.

\begin{corollary} \label{critspec}
Let $K$ be a field of characteristic $\ell \geq 0$, and let $\phi \in K(z)$ be bicritical of degree $d \geq 2$ with critical points $\gamma_1, \gamma_2 \in \PKbar$. Assume that $\ell = 0$ or $\ell > d$. If $K(\gamma_1, \gamma_2) \neq K$, then one of the following holds:
\begin{enumerate}
    \item $\phi$ has no critical orbit relations;
    \item $\gamma_1$ and $\gamma_2$ are both pre-periodic with the same period and pre-period; \label{casetwo}
    \item neither $\gamma_1$ nor $\gamma_2$ is pre-periodic and there is some $n \geq 2$ with $\phi^n(\gamma_1) = \phi^n(\gamma_2)$. \label{casethree}
\end{enumerate}   
\end{corollary}
\begin{remark} Under the hypotheses of Corollary \ref{critspec}, it follows from Theorem \ref{critfield}  that $\{\phi^n(\gamma_1), \phi^n(\gamma_2)\}$ is a $\Gal(K(\gamma_1, \gamma_2)/K)$-orbit for each $n \geq 1$. Thus in case \eqref{casethree} of Corollary \ref{critspec} we have that the common value $\phi^n(\gamma_1) = \phi^n(\gamma_2)$ lies in $K$ (cf. \cite[Section 3.2]{benedettocubic}). This does indeed occur, e.g. for $K = \Q$ and $\phi(z) = (z^2 + 2)/(z^2 + 2z + 2)$, which has $\{\gamma_1, \gamma_2\} = \{\pm \sqrt{2}\}$ and $\phi^2(\gamma_1) = \phi^2(\gamma_2) = 2/3$. When $d = 2$ and $\phi$ satisifes \eqref{casethree} of Corollary \ref{critspec}, the Galois theory of iterated preimages is studied in \cite{benedettocollision}.
\end{remark}

\begin{proof}
Let $\phi^n(\gamma_i) = \phi^m(\gamma_j)$ be a critical orbit relation for $\phi$. By Theorem \ref{critfield} there is $\sigma \in \Gal(K(\gamma_1, \gamma_2)/K)$ with $\sigma(\gamma_1) = \gamma_2$. 
Suppose that the critical relation describes a single critical orbit (i.e. $i = j$), and without loss say that it takes the form 
\begin{equation} \label{orbrelation}
\phi^s(\gamma_1) = \phi^t(\gamma_1)
\end{equation}
for $s > t \geq 0$ with $s$ and $t$ minimal. Applying $\sigma$ to both sides gives $\phi^s(\gamma_1) = \phi^t(\gamma_1)$, 
and $s$ and $t$ must again be minimal in this context, for otherwise we could apply $\sigma$ again and contradict the minimality of $s$ and $t$ in \eqref{orbrelation}.

Suppose now that the critical orbit relation is of the form $\phi^n(\gamma_1) = \phi^m(\gamma_2)$ for $n \geq m \geq 0$ with both $n$ and $m$ minimal. We may apply $\sigma$ to obtain $\phi^n(\gamma_2) = \phi^m(\gamma_1)$. Substitution then yields $\phi^{2n-m}(\gamma_1) = \phi^m(\gamma_1)$, and if $n > m$ then we are in the case of the previous paragraph. If $n = m$ and $\phi^n(\gamma_1)$ is pre-periodic, then we are again in the case of the previous paragraph. If $n = m$ and $\phi^n(\gamma_1)$ is not pre-periodic, then neither $\gamma_1$ nor $\gamma_2$ is pre-periodic. Note that $n=m$ ensures $n \geq 2$, as $\phi(\gamma_1) = \phi(\gamma_2)$ is ruled out by Corollary \ref{cor: RH}.
\end{proof}

We now give a normal form for any bicritical rational function. It is cleaner than, for instance, the form in \eqref{quadratic normal form}, but comes at the price of allowing conjugacy over $K(\gamma_1, \gamma_2)$. This form is nearly unique up to conjugacy in $\PGL_2(\Kbar)$. It is a generalization of the normal form used in \cite[Appendix A]{BJKL} for quadratic rational functions, which drew on work of Pink \cite{pink1}. Milnor \cite[Section 1]{milnorbicrit} gives a related normal form and discusses the moduli space of bicritical rational maps of degree $d$ over $\mathbb{C}$. 

\begin{theorem} \label{normalform}
Let $K$ be a field of characteristic $\ell \geq 0$, and let $\phi \in K(z)$ be bicritical of degree $d \geq 2$ with critical points $\gamma_1, \gamma_2 \in \PKbar$. Assume that $\ell = 0$ or $\ell > d$. 
\begin{enumerate}
\item If $\phi(\gamma_1) = \gamma_1$ and $\phi(\gamma_2) = \gamma_2$, then $\phi$ is conjugate over $K(\gamma_1, \gamma_2)$ to $cz^d$ for some $c \in K \setminus \{0\}$.
\item If $\phi(\gamma_1) = \gamma_2$ and $\phi(\gamma_2) = \gamma_1$, then $\phi$ is conjugate over $K(\gamma_1, \gamma_2)$ to $c/z^d$ for some $c \in K \setminus \{0\}$.
\item \label{typical} Otherwise, $\phi$ is conjugate over $K(\gamma_1, \gamma_2)$ to
\begin{equation} \label{mainform}
\frac{z^d + a}{z^d + b}
\end{equation}
for some $a, b \in K(\gamma_1, \gamma_2)$ with $a \neq b$.
\end{enumerate}
If two maps $(z^d + a)/(z^d + b)$ and $(z^d + a_1)/(z^d + b_1)$ in $\Kbar(z)$ are conjugate over $\Kbar$, then the conjugacy is either by $\mu(z) = z$ or $\mu(z) = (a/b)z$. In the latter case we have $a_1 = a^d/b^{d+1}$ and $b_1 = a^{d-1}/b^d$.
\end{theorem}

\begin{remark}
    The family of degree-$d$ polynomials is somewhat disguised in the normal form \eqref{normalform}. It is the family with $a = 0$, which has the totally ramified fixed point at $0$ instead of the usual $\infty$.
\end{remark}

\begin{remark}
The hypothesis $\ell = 0$ or $\ell > d$ in Theorem \ref{normalform} cannot be weakened to $\ell \nmid d$. For example, if $0 < \ell < d$ then $\phi(z) = z^d + z^\ell$ is bicritical with $e_0(\phi) = \ell$ and $e_\infty(\phi) = d$. Hence $\phi$ is not conjugate over $\Kbar$ to any of the maps given in Theorem \ref{normalform}. For another example, if $0 < \ell < d-1$ then 
\begin{equation} \label{univalue}
\phi(z) = \frac{z^\ell}{z^d + 1}
\end{equation}
is bicritical with $e_0(\phi) = \ell$ and $e_\infty(\phi) = d-\ell$, and hence not conjugate over $\Kbar$ to any of the maps given in Theorem \ref{normalform}. Note that maps of the form \eqref{univalue} have only one critical value, which is impossible in the cases $\ell = 0$ or $\ell > d$ by Corollary \ref{cor: RH}.
\end{remark}



\begin{proof}
    Let $\phi(z) = p(z)/q(z)$, where $p,q \in K[z]$ have no common root in $\Kbar$. We first handle the case where $\{\gamma_1, \gamma_2\} = \{0, \infty\}$. Let $v_i = \phi(\gamma_i)$ for $i \in \{1, 2\}$, and assume that neither $v_1$ nor $v_2$ is $\infty$. By Corollary \ref{cor: RH} we have $e_0(\phi) = e_\infty(\phi) = d$ (which in particular implies $v_1 \neq v_2$). It follows that $p(z) - v_1q(z) \in K[z]$ is a constant times $z^d$, while $p(z) - v_2q(z) \in K[z]$ has no roots in $\Kbar$, and so is constant. Solving the linear equations for $p(z)$ and $q(z)$ gives 
    \begin{equation} \label{firstform}
    \phi(z) = \frac{c_1z^d + a}{c_2z^d + b}
    \end{equation}
    for $c_1, c_2, a, b \in K$ with $bc_1 \neq ac_2$ (because $\phi$ has degree $d$). The argument in the case that $v_1 = \infty$ or $v_2 = \infty$ is similar. 
    
    Suppose that $\{v_1, v_2\} = \{0, \infty\}$. Then $\phi$ either fixes both $0$ and $\infty$ or interchanges them. The former occurs if and only if $a = c_2 = 0$, which implies $\phi(z) = c_3z^d$ for $c_3 = c_1/b\in K \setminus \{0\}$. 
    The latter occurs if and only if $b = c_1 = 0$, giving $\phi(z) = c_3z^{-d}$ for $c_3 = a/c_2 \in K \setminus \{0\}$. 
    
    When $\{v_1, v_2\} \neq \{0, \infty\}$, replace $\phi(z)$ by $1/\phi(1/z)$ if necessary so that  $\phi(\infty) \not\in \{0, \infty\}$. Thus $c_1c_2 \neq 0$ in \eqref{firstform}. Divide numerator and denominator by $c_2$ to get $\phi(z) = (c_3z^d + a')/(z^d + b')$, where $c_3 = c_1/c_2$, $a' = a/c_2$ and $b' = b/c_2$ are all in $K$. Now conjugate by $\mu(z) = (1/c_3)z$ to produce the map $(z^d + a'')/(z^d + b'')$ where $a'' = a'/c_3^{d+1}$ and $b'' = b'/c_3^d$. Note that conjugation by $\mu$ ensures that the conjugated map sends $\infty$ to $1$. Because $bc_1 \neq ac_2$, we have $a'' \neq b''$.

In the general case, let $\mu \in \PGL_2(\Kbar)$ satisfy $\mu(\gamma_1) = 0$ and $\mu(\gamma_2) = \infty$, and observe that $\phi$ can be taken to be in $\PGL_2(K(\gamma_1, \gamma_2))$. The map $\mu \circ \phi \circ \mu^{-1}$ is then bicritical with critical points $\{0, \infty\}$, and we have reduced to the previous case. 
    
    To prove the last assertion of the Theorem, let $\phi(z) = (z^d + a)/(z^d + b)$ and $\phi_1(z) = (z^d + a_1)/(z^d + b_1)$ be conjugate over $\Kbar$. Both maps are bicritical with critical points $0$ and $\infty$, so any $\Kbar$-conjugacy must preserve $\{0, \infty\}$. Hence the two maps must be conjugate by $\mu(z) = cz$ or $\mu(z) = c/z$ for some $c \in \Kbar$. It is then straightforward to check that in the former case we must have $c = 1$ and in the latter case we must have $c = (a/b)$. 
\end{proof}





Denote by $\textbf{U}_d$ the group of $d$th roots of unity in $\overline{K}$, and by $C_d$ the cyclic group of order $d$. We write $[C_d]^\infty$ for the infinite iterated wreath product of $C_d$. Recall that a labeling of the tree $T^d_\infty$ given by $\bigsqcup_{n \geq 0} \phi^{-n}(\alpha)$ is a choice of isomorphism $X^* \to T^d_\infty$, where $X^*$ is the abstract complete $d$-ary rooted tree whose vertices are words in $\{0, 1, \ldots, d-1\}$. There is a natural inclusion $[C_d]^\infty \hookrightarrow \Aut(X^*)$ as given in \cite[Section 2]{adamshyde}. 

\begin{corollary} \label{wreath}
 With hypotheses as in Theorem \ref{normalform}, let $K_0 = K(\gamma_1, \gamma_2)$, and assume that $\alpha \in K_0$ is not post-critical. For $n \geq 1$, let $K_n(\alpha) = K_0(\phi^{-n}(\alpha))$ and $G_n(\alpha) = \Gal(K_n(\alpha)/K_0)$. Let $K_\infty(\alpha) = \bigcup_{n \geq 1} K_n(\alpha)$, and $G_\infty(\alpha) = \Gal(K_\infty(\alpha)/K_0))$.
 
 Then 
 $\textbf{U}_d \subseteq K_1( \alpha)$, $\textbf{U}_d$ acts by multiplication on $\phi^{-n}(\alpha)$ for each $n \geq 1$, and the orbits of this action are the fibers of the projection $\phi^{-n}(\alpha) \to \phi^{-(n-1)}(\alpha)$ given by applying $\phi$. 
 In particular, for each $n \geq 2$ the quotient $G_n(\alpha)/G_{n-1}(\alpha)$ embeds into $(C_d)^{d^{n-1}}$. 
 
 If we further assume that $\textbf{U}_d \subseteq K_0$, then there is a labeling of $T^d_\infty$ so that 
 \begin{equation} \label{injection}
      G_\infty(\phi, \alpha) \hookrightarrow [C_d]^\infty \subseteq \Aut(X^*).
 \end{equation}
\end{corollary}

\begin{proof}
    It follows from Theorem \ref{normalform} that $\phi$ is conjugate by $\mu \in \PGL_2(K_0)$ to 
    $$
     \phi_1(z) = \frac{c_1z^d + a}{c_2z^d + b},
    $$
    where $c_1, c_2, a, b \in K_0$ and $bc_1 \neq ac_2$.
    Because $\mu \in \PGL_2(K_0)$, 
    we have 
    $$\mu(\phi^{-n}(\alpha)) =  (\phi^\mu)^{-n}(\mu(\alpha)) = \phi_1^{-n}(\mu(\alpha))$$
   for all $n \geq 1$, and thus $K_0(\phi^{-n}(\alpha)) = K_0(\phi_1^{-n}(\mu(\alpha)))$. Hence it suffices to prove the theorem for $\phi_1$ with basepoint $\mu(\alpha)$. To ease notation we continue to use $\alpha$ instead of $\mu(\alpha)$. 

    Because $\alpha$ is not post-critical, $T_\infty(\phi_1, \alpha) \cap \{0, \infty\} = \emptyset$, and hence for any $\beta \in T_\infty(\phi_1, \alpha)$ we have that $\phi_1^{-1}(\beta)$ consists of the $d$ distinct solutions to $\gamma^d = C$, where $C = \frac{b\beta - a}{c_1 - c_2\beta} \in K \setminus \{0\}$. Taking $\beta = \alpha$ then gives that $\textbf{U}_d \subseteq K_1(\phi, \alpha)$. For each $\gamma \in \phi^{-n}(\alpha)$ and each $\zeta_d^m \in \textbf{U}_d$, we have $\phi(\gamma) = \phi(\zeta_d^m \gamma)$, proving the assertions about the action of $\textbf{U}_d$ on $\phi^{-n}(\alpha)$. Because $\textbf{U}_d \subseteq K_1(\phi, \alpha)$, the action of $G_n(\phi, \alpha)/G_{n-1}(\phi, \alpha)$ on each fiber of the map $\phi^{-n}(\alpha) \to \phi^{-(n-1)}(\alpha)$ is cyclic, proving that this group embeds into $ (C_d)^{\#\phi^{-(n-1)}(\alpha)} = (C_d)^{d^{n-1}}$.

   Finally, when $\textbf{U}_d \subseteq K_0$, one establishes \eqref{injection} using the same labeling constructed in the setting of iterated monodromy groups (where $\alpha$ is replaced by an element transcendental over $K_0$). See e.g. \cite[Section 3]{adamshyde}.
\end{proof}


We record the following consequence of Corollary \ref{critspec}, Theorem \ref{normalform}, and Corollary \ref{wreath}.

\begin{corollary} \label{gooddef}
Let hypotheses and notation be as in Corollary \ref{wreath}. If $\phi$ has a trailing critical point, or precisely one critical point of $\phi$ is pre-periodic, or both critical points of $\phi$ are pre-periodic but with different period or pre-period, then $K_0 = K$ in Corollary \ref{wreath}, and there is $\mu \in \PGL_2(K)$ such that $\phi^\mu$ has one of the forms given in Theorem \ref{normalform}. In particular, $K_n(\phi, \alpha) = K_n(\phi^\mu, \mu(\alpha))$ for all $n \geq 1$. 
%
\end{corollary}




\section{Rigid divisibility and Irreducibility results} 

\label{sec: rigid}

Here we prove that, under a good reduction assumption, rational functions without linear terms lead to rigid divisibility sequences. Similar phenomena have been noted for monic polynomials with integer coefficients (e.g. \cite[Proposition 3.2]{rice}, \cite[Lemma 5.3]{quaddiv}). We wish to generalize these results to rational functions, and allow the coefficients to be in the field of fractions of an arbitrary Dedekind domain. 

In our definition of rigid divisibility, we allow for a finite set of ``bad" primes, as in the definition used in \cite{quaddiv}. By a complete set of non-archemedean absolute values on a field $K$, we mean a set containing one choice from each equivalence class. 

\begin{definition}
\label{def:rigid}
Let $K$ be a field and $A$ be a complete set of non-archemedean absolute values on $K$. Let $(c_i)_{i\geq 1}$ be a sequence in $K$ and let $S \subseteq A$ be finite. We say $(c_i)_{i\geq 1}$ is an \emph{$S$-rigid divisibility sequence} if for all $|\cdot| \in A \setminus S$, the following hold:
\begin{enumerate}
    \item If $|c_n|<1$, then $|c_n| = |c_{kn}|$ for all $k\geq 1$.
    \item If $|c_m|<1$ and $|c_n|<1$, then $|c_{\gcd(m,n)}|<1$.
\end{enumerate}
We call $(c_i)_{i \geq 1}$ a \textit{rigid divisibility sequence} if it is an $S$-rigid divisibility sequence with $S = \emptyset$.
\end{definition}

For the rest of this section, we work in the setting where $K$ is the field of fractions of a Dedekind domain $R$. First note that if $\phi = p/q$ for $p,q \in K[z]$, then each coefficient $a$ of either $p$ or $q$ generates a fractional ideal in $K$. Thus there is $c \in R$ with $c(a) \subset R$, implying $ca \in R$. Letting $C$ be the product of such $c$ as we vary over all coefficients of $p$ or $q$, we have $Cp, Cq \in R[z]$. Hence we may assume without loss that $\phi$ is the quotient of two polynomials with coefficients in $R$.

In our main result (Theorem \ref{thm:rigid}), we consider $\phi = p/q \in R(z)$ and take the set $S$ to be the absolute values where either $\phi$ has bad reduction or the pair $p,q$ is not normalized (see Section \ref{subsection: normalized and reduction} for definitions). It's thus desirable to replace the pair $p, q$ with constant multiples $cp, cq$ that are normalized with respect to as many absolute values as possible. In some cases $p,q$ is normalized with respect to every $|\cdot| \in A$; for instance, when at least one of the coefficients of $\phi$ is a unit in $R$. In the case where $R$ is a principal ideal domain, we obtain a similar result for any rational function:

\begin{proposition} \label{prop:PID}
Let $K$ be the field of fractions of a principal ideal domain $R$, and fix a complete set $A$ of non-archemedean absolute values on $K$. Let $p, q \in R[z]$. Then there is $c \in K$ such that the pair $cp, cq$ is normalized with respect to every $|\cdot| \in A$.
\end{proposition}

\begin{proof}
    Let $\{a_1, \ldots, a_n\}$ be the union of the non-zero coefficients of $p$ and the non-zero coefficients of $q$. 
    The set $\{|\cdot| \in A : \text{$|a_i| < 1$ for all $i$}\}$ is finite, and we write it 
    $\{|\cdot|_1, \ldots, |\cdot|_t\}$. For each $j \in \{1, \ldots, t\}$, let 
    $$
    m_j = \max_{1 \leq i \leq n}\{|a_i|_j \} < 1.
    $$
    Because $R$ is a principal ideal domain, we may select $\pi_j \in R$ with $|\pi_j|_j = m_j$ and $|\pi_j| = 1$ for all other $|\cdot| \in A$. Taking $\pi = \pi_1 \pi_2 \cdots \pi_t$, we have that $(1/\pi)p, (1/\pi)q$ is normalized with respect to every $|\cdot| \in A$.
\end{proof}



To motivate our main result, we give an example that illustrates the application we will make in Section \ref{Galois}.

\begin{example} \label{ex:rigid}
Let $\phi(z) = \frac{z^2 + 1}{z^2 + 3} \in \Q(z)$. Note that $\phi$ is normalized with respect to every non-archimedean absolute value on $\Q$, has no linear term in either numerator or denominator, and has bad reduction only at the $2$-adic absolute value. Eq. \ref{eq:p_n,q_n} gives $p_1(z) = z^2 + 1$, $q_1(z) = z^2 + 3$, and for $n \geq 2$, 
$$
p_n = p_{n-1}^2 + q_{n-1}^2, \qquad q_n = p_{n-1}^2 + 3q_{n-1}^2.
$$
The prime factorizations of the first 8 terms of the sequence $(p_n(0))_{n \geq 1}$ are 
$$\begin{array}{l}
 1   \\
2  \cdot 5  \\
2^2 \cdot 13 \cdot 17 \\
2^5  \cdot 5 \cdot 42461 \\ 
2^{10} \cdot 109 \cdot 13337 \cdot 268897 \\
2^{21} \cdot 5 \cdot 13 \cdot 17 \cdot 193 \cdot 11969 \cdot 3144217 \cdot 82530809 \\
2^{42} \cdot 157 \cdot 15170009 \cdot \ell_1 \\
2^{85} \cdot 5 \cdot 521 \cdot 7297 \cdot 7841 \cdot 42461 \cdot 697121 \cdot 207272581 \cdot \ell_2,
\end{array}
$$
where $\ell_1 \approx 10^{39}$ and $\ell_2 \approx 10^{68}$.
\end{example}

We now prove Theorem \ref{intro: rigid}, which we restate here for the convenience of the reader.

\begin{theorem}
\label{thm:rigid}
Let $K$ be the field of fractions of a Dedekind domain $R$, and fix a complete set $A$ of non-archemedean absolute values on $K$. Let $\phi \in K(z)$ have degree $d \geq 2$, and let $p,q \in R[z]$ be relatively prime polynomials with $\phi = p/q$.
Let $p_n$ and $q_n$ be as in Eq. \ref{eq:p_n,q_n}, and let $S \subseteq A$ be the set of absolute values at which either $\phi$ has bad reduction or the pair $p,q$ is not normalized. Assume that 
$$
p'(0) = q'(0) = 0.
$$
Then the sequence $(p_n(0))_{n\geq 1}$ is an $S$-rigid divisibility sequence.
\end{theorem}
\begin{proof}
Let $|\cdot| \in A \setminus S$. Because the pair $p, q$ is normalized with respect to $|\cdot|$ and $\phi$ has good reduction at $|\cdot|$, Proposition \ref{prop: reduction} gives that $\max\{|p_i(0)|, |q_i(0)|\}=  1$ for each $i \geq 1$. Thus $|\phi^i(0)| < 1$ is equivalent to $|p_i(0)| < 1$, and when either holds we have $|\phi^i(0)| = |p_i(0)|$ and $|q_i(0)| = 1$.

To prove that $(p_n(0))_{n \geq 1}$ is an $S$-rigid divisibility sequence, we assume $|\phi^n(0)| < 1$ and prove $|\phi^{nk}(0)| = |\phi^n(0)|$ for all $k \geq 1$ by inducting on $k$. The case $k = 1$ is true by hypothesis, so assume $|\phi^{n(k-1)}(0)| = |\phi^n(0)|$. 
Using the chain rule on the recursions in Eq. \ref{eq:p_n,q_n} and the hypothesis that $p'(0) = q'(0) = 0$, we have that $p_n'(0) = q_n'(0) = 0$. Write $p_n(z) = z^2f(z) + p_n(0)$, where $f \in R[z]$. 
We have 
$$|\phi^{n(k-1)}(0)| = |\phi^n(0)| = |p_n(0)| < 1,$$
whence $|\phi^{n(k-1)}(0)|^2 < |p_n(0)|$. Because $|\cdot|$ is non-archemedean and $f$ has coefficients in $R$, it follows that
\begin{equation} \label{strong}
|p_{n}(\phi^{n(k-1)}(0))| = |(\phi^{n(k-1)}(0))^2f(\phi^{n(k-1)}(0)) + p_n(0)| = |p_n(0)|. 
\end{equation}
Our assumption that $|\phi^n(0)| < 1$ implies $|q_n(0)| = 1$. Using a similar argument to the one giving Eq. \ref{strong}, we find $|q_{n}(\phi^{n(k-1)}(0))| = 1$. Therefore
\begin{equation*} \label{first}
|\phi^{nk}(0)| = |\phi^n(\phi^{n(k-1)}(0))| = \frac{|p_n(\phi^{n(k-1)}(0))|}{|q_n(\phi^{n(k-1)}(0))|} = |p_n(0)| = |\phi^n(0)|,
\end{equation*}
 as desired.

To complete the proof, suppose $|\phi^m(0)|<1$ and $|\phi^n(0)|<1$ for some $m, n \geq 1$. Take the smallest positive integer $t$ with $|\phi^t(0)|<1$. We claim that $t \mid m$ and $t \mid n$ must hold. Write $m = tq + r$ for $0 \leq r \leq t-1$, and let $\tilde{\phi}$ be the reduction of $\phi$. Observe that because $\phi$ has good reduction at $|\cdot|,$ $|\phi^{i}(0)| < 1$ is equivalent to $\tilde{\phi}^i(0) = 0$ for any $i \geq 1$. Thus 
$$
0 = \tilde{\phi}^m(0) = \tilde{\phi}^r(\tilde{\phi}^{qt}(0)) = \tilde{\phi}^r(0),
$$
which contradicts the minimality of $t$ unless $r = 0$. Hence $t \mid m$, and a similar argument gives $t \mid n$. Therefore $t \mid g$, where $g = \gcd(m,n)$. It follows that $\tilde{\phi}^g(0) = 0$, whence $|\phi^g(0)| < 1$, as desired.
\end{proof}

\begin{remark}
    Theorem \ref{thm:rigid} requires that $p_n$ have the argument of $0$ because $|p_n(0)| < 1$ occurs precisely when $\tilde{\phi}^n(0) = 0$, implying $0$ is periodic under $\tilde{\phi}$ with period dividing $n$. The sequence $(q_n(0))_{n \geq 1}$ is not in general an $S$-rigid divisibility sequence because $|q_n(0)| < 1$ occurs when $\tilde{\phi}^n(0) = \infty$, giving no periodicity information on an orbit of $\tilde{\phi}$. 

\end{remark}

We turn now to irreducibility results, which will be useful in Section \ref{Galois} in proving that $G_\infty(\phi, \alpha) \cong \Aut(T_\infty(\phi, \alpha))$.  One is, however, often interested in establishing only $[\Aut(T_\infty(\phi, \alpha)) : G_\infty(\phi, \alpha)] < \infty$, and for this purpose it is enough to know that $(\phi, \alpha)$ is eventually stable over $K$, that is, that the number of irreducible factors of $p_n(z) - \alpha q_n(z)$ (or just $q_n(z)$ when $\alpha = \infty$) is bounded. The most general result on this topic known to the authors is a consequence of \cite[Theorem 4.6]{evstab}.  We state it here in the setting of bicritical rational functions, using the normal form of case (3) of Theorem \ref{normalform}. We note that eventual stability of $(\phi, \alpha)$ in the case $\phi(z)=  cz^d$ and $\phi(z) = c/z^d$ is well-understood (see e.g. \cite[Section 3.4]{evstab}). 

\begin{theorem}
 Let $K$ be a field of characteristic $\ell \geq 0$, and let 
 $$
 \phi(z) = \frac{z^d + a}{z^d + b} \in K(z),
 $$
where $a \neq b$, $d \geq 2$, and either $\ell = 0$ or $\ell > d$. Assume that $\alpha \in \PP^1(K)$ is not a periodic point for $\phi$. Let $|\cdot|$ be a non-archimedean absolute value on $K$ whose residue field $k$ has characteristic $p \geq 0$. Suppose that one of the following holds:
\begin{enumerate}
    \item $k$ is finite, $d$ is a power of $p$, $|a| \leq 1, |b| \leq 1$, and $|a-b| = 1$.
    \item $|a| < 1$, $|b| = 1$, and $|\alpha| < 1$.
\end{enumerate} 
Then $(\phi, \alpha)$ is eventually stable over $K$.
\end{theorem}

\begin{proof}
    In Case (1), the assumptions on $|a|$, $|b|$, and $|a-b|$ imply that $\phi$ has good reduction at $|\cdot|$. By \cite[Proposition 4.9]{evstab}, the hypotheses on $k$ and $d$ imply that $\phi$ is bijective on residue extensions. The eventual stability of $(\phi, \alpha)$ then follows from \cite[Corollary 4.10]{evstab}. 
    
    In Case (2), the assumptions on $|a|$ and $|b|$ imply that $\tilde{\phi} = z^d/(z^d + \tilde{b})$. Thus this is the case of good, polynomial reduction. Note that $\tilde{\phi}^{-1}(0) = \{0\}$ and $|\alpha| < 1$ implies $\tilde{\alpha} = 0$. The eventual stability of $(\phi, \alpha)$ then follows from \cite[Theorem 4.6]{evstab}. 
\end{proof}

For the results of Section \ref{Galois} we need the stronger result that $\phi$ is stable, i.e. $p_n(z) - \alpha q_n(z)$ (or $q_n(z)$ if $\alpha = \infty$) is irreducible over $K$ for all $n \geq 1$. We restrict to the case that $\phi$ has degree 2 and $\alpha = 0$. Our main tool is  \cite[Lemma 3.4]{galrat}: 

\begin{lemma}[\cite{galrat}]
\label{lemma:galrat-stability}
Let $K$ be a number field and let $\phi=p(z)/q(z)\in K(z)$ have degree $2$, where $p, q \in K[z]$ are relatively prime. Let $\gamma_1, \gamma_2\in \mathbb{P}^1(\overline{K})$ be the critical points of $\phi$, let $p_n, q_n\in K[z]$ be as in Eq. \ref{eq:p_n,q_n}, and let $K_n = K(\phi^{-n}(0))$ be the splitting field of $p_n$. Assume that $p_{n-1}$ is irreducible in $K[z]$, and let $\beta \in \overline{K}$ be a root of $p_{n-1}$.

Then there exists $C\in K$ such that $p_n$ is irreducible in $K[z]$ if and only if
\[C\prod_{\phi(\gamma_i)\neq \infty} (\phi(\gamma_i)-\beta)\notin K_{n-1}^{\times 2},\]
where $K_{n-1}^{\times 2}$ denotes the set of non-zero squares in $K_{n-1}$.
If $p$ has distinct roots in $\overline{K}$, then we may take 
\[C = \text{Disc}(p) \cdot \prod_{\phi(\gamma_i)\neq \infty} \phi(\gamma_i)^{-1},\]
where $\text{Disc}(p)$ is the discriminant of $p$.
\end{lemma}

We focus for the rest of this section on the family 
\begin{equation} \label{eq:mainfamily}
\phi=\frac{z^2+a}{z^2}\in K(z). 
\end{equation}
Thus $p = p_1 = z^2 + a$, $q = q_1 = z^2$, and from Eq. \ref{eq:p_n,q_n} we have for $n \geq 2$,
\begin{equation}
\label{eq:p_n,q_n mainfamily}
p_n = p_{n-1}^2+aq_{n-1}^2,\qquad q_n = p_{n-1}^2.
\end{equation}
Applying Lemma \ref{lemma:galrat-stability} to this family gives:
\begin{lemma}
\label{lemma:stability}
Let $K$ be a number field, let $\phi = (z^2 + a)/z^2 \in K[z]$ with $a \neq 0$, and let $p_n, q_n \in K[z]$ be as in Eq. \ref{eq:p_n,q_n mainfamily}. 
Suppose that $p_{n-1}$ is irreducible in $K[z]$ for some $n\geq 2$. Then $p_n$ is irreducible in $K[z]$ if $p_{n-1}(1)$ is not a square in $K$.
\end{lemma}
\begin{proof}
Let $K_n$ be the splitting field of $p_n$ and $\beta \in \overline{K}$ a root of $p_{n-1}$. The critical points of $\phi$ are $\gamma_1=0$ and $\gamma_2=\infty$, with $\phi(0) = \infty$ and $\phi(\infty) =1$. We have $p=z^2+a$, which has distinct roots in $\overline{K}$ since $a \neq 0$ and $K$ has characteristic $\neq 2$. Since $\phi(\gamma_1)=\infty$, with the notation of Lemma~\ref{lemma:galrat-stability}, we obtain $C = -4a(\phi(\infty)^{-1}) = -4a$. From Lemma~\ref{lemma:galrat-stability} we conclude that
\[p_n\text{ is irreducible in $K[z]$ if and only if }-4a(1-\beta)\notin K_{n-1}^{\times 2}.\]
Note that $-a$ (hence $-4a$) is already a square in $K_1$, which gives $-4a\in (K_{n-1}^{\times})^2$ since $K_1\subseteq K_{n-1}$. Hence, we conclude
\[p_n\text{ is irreducible in }K[z]\text{ if and only if }1-\beta \notin (K_{n-1}^{\times})^2.\]
Let $\beta_1, \beta_2, \dots, \beta_k$ be the roots of $p_{n-1}$ in $\overline{K}$, where we assume $\beta_1 = \beta$.
Observe from Eq. \ref{eq:p_n,q_n mainfamily} that $p_{n-1}$ is monic. This combined with the irreducibility of $p_{n-1}$ gives
\[\text{Nm}_{K_{n-1}/K}(1-\beta) = \prod_{i=1}^{k} (1-\beta_i) = p_{n-1}(1).\]
Because $1-\beta$ cannot be a square in $K_{n-1}$ if $\text{Nm}_{K_{n-1}/K}(1-\beta)$ is not a square in $K$, the proof is complete.
\end{proof}

\section{Galois-theoretic results over $\Q$} \label{Galois}

In this section we consider the family in Eq. \ref{eq:mainfamily} with ground field $K = \Q$ and basepoint $\alpha = 0$. We write $K_n$ for $K_n(0)$, $G_n$ for $G_n(0)$, and $T_\infty$ for $T^2_\infty$. Note that the map $G_\infty \hookrightarrow \Aut(T_\infty)$ in Eq. \ref{intro: wreathprod eqn}. is an isomorphism if and only if $[K_n : K_{n-1}] = 2^{2^{n-1}}$ for all $n \geq 1$. 

\begin{theorem} \label{thm:surjectivityA}
    Let $\phi(z) = \frac{z^2+a}{z^2},$
    where $a \in \Z$ satisfies $a \equiv 2 \pmod{4}$ and $a \leq -3$. If $n \geq 1$ is not square-free, then $[K_n : K_{n-1}] = 2^{2^{n-1}}$.
\end{theorem}

\begin{theorem}
\label{thm:surjectivityB}
Let $m\notin \{-1,0,1\}$ be an integer, and set 
$$\phi(z) = \frac{z^2-2(2m^2-1)^2}{z^2}.$$ 
Let $P$ be the set of all prime numbers, and define the sets $S_1$ and $S_2$ by 
\[S_1=\{p\in P\text{ }|\text{ }p\equiv 3\text{ }(\text{mod }4)\},\text{ }S_2 = \{p\in P\text{ }|\text{ }p\equiv 5\text{ or }7\text{ }(\text{mod }8)\}.\]
Suppose that $m \in \Z$ satisfies the following conditions:
\begin{enumerate}
    \item $m$ is congruent to $-1,0$ or $1$ modulo a prime in $S_1$.
    \item $2m$ is congruent to $-1$ or $1$ modulo a prime in $S_2$.
\end{enumerate}
Then $G_\infty \cong \text{Aut}(T_\infty)$. 
\end{theorem}
\begin{remark}
It is easy to verify that any positive integer $m$ such that $m \not\equiv 1 \pmod{4}$ satisfies the conditions of Theorem \ref{thm:surjectivityB}. 
\end{remark}


\medskip

\noindent \textbf{Notation.} Throughout this section, we use the following notation, where $\alpha \in \Q$, $p = p_1 = z^2 + a$ for $a \in \Z$, $q = q_1 = z^2$, and $p_i, q_i$ are as in Eq. \ref{eq:p_n,q_n mainfamily}. 
\begin{itemize}
    \item $\overline{\alpha}$ denotes the image of $\alpha$ in $\mathbb{Q}^{\times}/{\mathbb{Q}^{\times2}}$

    \item $|\cdot|_\ell$ and $v_\ell$ denote the $\ell$-adic absolute value and $\ell$-adic valution for a prime $\ell$

    \item for $n \geq 1$, $\beta_{\alpha, n}$ denotes the product $\prod_{d \mid n}(p_d(\alpha))^{\mu(n/d)}$, where $\mu$ is the M\"obius function \label{betadef}

    \item $\beta_n$ denotes $\beta_{0,n}$

    \item $(f_n)_{n \geq 1}$ denotes the integer sequence defined by $f_1 = f_2 = 1$, and for $n \geq 3$, 
\begin{equation} \label{fngn}
    f_n = (f_{n-1})^2 + a(f_{n-2})^4
\end{equation}
\item for $n \geq 1$, $\theta_n$ denotes the product $\prod_{d \mid n}(f_d)^{\mu(n/d)}$
\end{itemize}

\medskip


The sequences $(\beta_n)_{n \geq1}, (\theta_n)_{n \geq 1}$, and $(f_n)_{n \geq 1}$ will play key roles in the proofs of Theorems \ref{thm:surjectivityA} and \ref{thm:surjectivityB}.
We begin with a result giving properties of $(f_n)_{n \geq 1}$.

\begin{lemma}
\label{power of a}
Assume $a \neq 0$. For all $n \geq 1$ we have: 
\begin{enumerate}
    \item $p_n(0) = a^{2^{n-1}}f_n$,
    \item $f_n \equiv 1 \pmod{|a|}$,
    \item $(f_n)_{n \geq 1}$ is a rigid divisibility sequence. 
\end{enumerate}
\end{lemma}
\begin{proof}
We have $p_1(0) = a$ and $p_2(0) = a^2$ so condition (1) holds for $n = 1, 2$. 
Suppose that it holds for fixed $N\geq 2$. From Eq. \ref{eq:p_n,q_n mainfamily} and Eq. \ref{fngn} we obtain
\begin{align*}
p_{N+1}(0) & = (p_N(0))^2+a(q_N(0))^2 \\
& = (p_N(0))^2+a(p_{N-1}(0))^4 \\ 
& = a^{2^N}(f_N)^2 + a^{2^N + 1}(f_{N-1})^4 \\
&= 
 a^{2^N}f_{N+1},
\end{align*}
proving condition (1).
Condition (2) holds by a straightforward induction. To prove condition (3), note that Theorem \ref{thm:rigid} and the fact that $\text{Res}(p, q) = a^2$ show that if $\ell \nmid a$, then the $\ell$-adic absolute value $|\cdot|_\ell$ satisfies the conditions of Definition \ref{def:rigid} for the sequence $(p_n(0))_{n \geq 1}$. For such $\ell$, condition (1) of the present lemma shows $|p_n(0)|_\ell = |f_n|_\ell$ for all $n \geq 1$, whence the same conclusion holds for the sequence $(f_n)_{n \geq 1}$. On the other hand, if $\ell \mid a$, then condition (2) of the present lemma shows $|f_n|_\ell = 1$ for all $n$, and thus the conditions of Definition \ref{def:rigid} hold as well for the sequence $(f_n)_{n \geq 1}$ and the absolute value $|\cdot|_\ell$.
\end{proof} 

In the next lemma, we describe the signs of $p_n(0)$ and $\beta_n$. 
Denote the sign of a non-zero real number $r$ by $\sgn(r) \in \{\pm 1\}$.
\begin{lemma}
\label{lemma:sign of p_n}
Suppose that $a \leq  - 3$, and recall 
that $\beta_n=\prod_{d|n}(p_d(0))^{\mu(n/d)}$. Then
\begin{enumerate}
   \item for all $n \geq 1$, $p_n(0) \neq 0$ and $\sgn(p_n(0)) = (-1)^n$
   \item $\sgn(\beta_n) = 1$ for all $n \geq 3$.
\end{enumerate}

\end{lemma}
\begin{proof}
Write $a=-b$ for $b \geq 3$, so that $\phi = \frac{z^2-b}{z^2}$. Observe that $\sgn(p_1(0))=\sgn(-b)=-1$, and thus (1) holds for $n = 1$. We have 
$$p_n(0) = \phi^n(0)q_n(0) = \phi^n(0)(p_{n-1}(0))^2$$ for any $n\geq 2$. 
Thus to prove (1) it suffices to show that $\sgn(\phi^n(0)) = (-1)^n$ for $n\geq 2$. Indeed, we will show that $\phi^n(0) > 0$ when $n \geq 2$ is even and $\phi^n(0) \leq 1 - b$ when $n \geq 2$ is odd. 
We have $\phi(z) = 1 - (b/z^2)$, and so $\phi(z_0) \leq 1$ for all $z_0 \in \mathbb{R}$. If $\phi(z_0) > 0$, then 
$$
\phi^2(z_0) = 1 - \frac{b}{\phi(z_0)^2} \leq 1-b.
$$
If $\phi(z_0) \leq 1-b$, then 
\[\phi^2(z_0) =  1 - \frac{b}{\phi(z_0)^2} \geq 1-\frac{b}{(1-b)^2} = \frac{b^2-3b+1}{(1-b)^2},\]
and this last expression is positive since $b \geq 3$. The desired statement now follows by induction and the fact that $\phi^2(0) = 1$ and $\phi(\phi^2(0)) = 1-b$.

To prove part (2) of the lemma, we begin by citing a special case of \cite[Lemma 2.4]{BG}: 
\begin{equation} \label{mob}
\sum_{\substack{d \mid n \\ \text{$d$ odd}}} \mu(n/d) = 0 \quad \text{for $n \geq 3$.}
\end{equation}
It follows from part (1) of the present lemma that
$$
\sgn(\beta_n)
= \prod_{\substack{d|n \\ \text{$d$ odd}}} (-1)^{\mu(n/d)},
$$
and the desired conclusion holds in light of Eq. \ref{mob}.
\end{proof}



We now quote two key results from \cite{galrat}, which will be crucial in the proof of Theorems~\ref{thm:surjectivityA} and \ref{thm:surjectivityB}. The first is a case of \cite[Theorem 3.2]{galrat}, which allows us to compute the polynomial discriminant of $p_n$. We denote the leading coefficient of a polynomial $f$ by $\ell(f)$.

\begin{theorem}[\cite{galrat}]
\label{thm:galrat-Disc}
Let $K$ be a number field, let $\phi = p(z)/q(z)\in K(z)$ have degree $d\geq 2$, where $p,q \in K[z]$ are relatively prime. Let $p_n,q_n$ be as in Eq. \ref{eq:p_n,q_n}, and let $c = qp'-pq'$. Assume that $\phi(\infty)\neq \infty$ and that for some $n \geq 2$, $\phi^n(\infty)\neq 0$ and $\phi^{n-1}(\infty)\neq 0$. Then
\begin{equation}
\label{eq:Disc-non-infty}
\text{Disc}(p_n) = \pm \ell(p_n)^{k_1}\ell(q)^{k_2}\ell(c)^{k_3}(\text{Disc}(p_{n-1}))^d\text{Res}(q,p)^{d^{n-1}(d^{n-1}-2)}\prod_{c(\gamma)=0}p_n(\gamma),
\end{equation}
where
\[k_1 = 2d-2-d_c,\text{ }k_2 = d^{n-1}(d-d_p)(d^{n-1}-2),\text{ and }k_3 = d^n.\]
\end{theorem}

The second result we require is \cite[Theorem 3.7]{galrat}, which gives a necessary and sufficient condition for $[K_n:K_{n-1}] = 2^{2^{n-1}}$ in the case where $\phi$ is quadratic.

\begin{theorem}[\cite{galrat}]
\label{thm:galrat-maximality}
Let $K$ be a number field, let $\phi = p(z)/q(z)\in K(z)$ have degree $2$, and let $p_n, q_n$ be as in Eq. \ref{eq:p_n,q_n}. Let $n \geq 2$, let $\ell(p_{n-1})$ be the leading coefficient of $p_{n-1}$, let $\gamma_1, \gamma_2\in \mathbb{P}^1(\overline{K})$ be the critical points of $\phi$, and without loss say $\phi(\gamma_1)\neq \infty$. Assume that $\phi^n(\infty)\neq 0$, $\phi^{n-1}(\infty)\neq 0$, and $p_{n-1}$ is irreducible in $K[x]$.  If $\phi(\gamma_2)$ is not (resp. is) $\infty$, then $[K_n:K_{n-1}] = 2^{2^{n-1}}$ if and only if
\begin{equation}
p_{n-1}(\phi(\gamma_1))p_{n-1}(\phi(\gamma_2))\notin K_{n-1}^{\times 2}\text{ }(\text{resp. }\ell(p_{n-1})p_{n-1}(\phi(\gamma_1))\notin K_{n-1}^{\times 2}).
\end{equation}
\end{theorem}

We now set about applying Theorems \ref{thm:galrat-Disc} and \ref{thm:galrat-maximality} to $\phi(z) = (z^2 + a)/z^2$. Let $\ell(p_n)$ denote the leading coefficient of $p_n$, and recall the definition of $(f_n)_{n \geq 1}$ and $(\theta_n)_{n \geq 1}$ on p. \pageref{betadef}.

\begin{lemma}
\label{lemma:ell(p_n)}
We have
\begin{enumerate}
\item $\ell(p_n) = f_{n+1}$ for any $n\geq 1$.
\item $p_n(1) = f_{n+2}$ for any $n\geq 1$.
\end{enumerate}
\end{lemma}
\begin{proof}
One computes $\ell(p_1) = 1 = f_2$ and $\ell(p_2) = a+1 = f_3$. Since $\deg(p_n) = \deg(q_n)$ for $n\geq 1$, Eq.~\ref{eq:p_n,q_n mainfamily} gives
\begin{align*}
\ell(p_n) = (\ell(p_{n-1}))^2 + a(\ell(q_{n-1}))^2
= (\ell(p_{n-1}))^2 + a(\ell(p_{n-2}))^4
\end{align*}
for $n\geq 3$. Condition (1) of the present lemma follows using induction and the definition of $f_n$. The proof of condition (2) is similar. 
\end{proof}

\begin{lemma}
\label{lemma:sufficient-theta_n} 
Suppose that for some $n \geq 2$, $p_{n-1}$ is irreducible over $\mathbb{Q}$. If $|\theta_{n+1}|$ is not a square in $\mathbb{Q}$, then $[K_n:K_{n-1}]=2^{2^{n-1}}$. 
\end{lemma} 

\begin{proof}
We apply Theorem \ref{thm:galrat-Disc} with $p(z) = z^2 + a$ and $q(z) = z^2$. Observe that the theorem applies because by condition (1) of Lemma \ref{lemma:sign of p_n} we have $\phi^i(0) \neq 0$ for all $i \geq 1$, and then from $\phi(0) = \infty$ it follows that $\phi^i(\infty) \neq 0$ for all $i$.
Using the notation of Theorem \ref{thm:galrat-Disc}, we have $c(z) = -2az$, and thus the only root of $c$ is zero (the other critical point of $\phi$ is $\infty$). 
We have $k_1=1$, $k_2=0$, $k_3=2^n$, $\ell(c)=-2a$, and $\text{Res}(q,p) = a^2$. Therefore, using Eq.~\ref{eq:Disc-non-infty} and simplifying, we obtain
\[\text{Disc}(p_n) = \pm 2^{2^n}a^{2^{2n-1}-2^n}(\text{Disc}(p_{n-1}))^2\ell(p_n)p_n(0).\]
Using part (1) of Lemma~\ref{power of a}, part (1) of Lemma~\ref{lemma:ell(p_n)}, and simplifying, we obtain
\begin{equation}
\label{eq:Disc(p_n)}
\text{Disc}(p_n) = \pm 2^{2^n}a^{2^{2n-1}-2^{n-1}}(\text{Disc}(p_{n-1}))^2f_{n+1}f_n.
\end{equation}
Suppose that there exists a rational prime $p$ such that 
\begin{equation}
\label{eq:odd valuation p}
\gcd(p,2)=\gcd(p,a)= 1, \; \; \gcd(p,f_i)=1 \text{ for all $i \leq n$, }\text{ and }v_p(f_{n+1})\text{ is odd},
\end{equation}
where $v_p$ denotes the $p$-adic valuation. We first show that the existence of such a prime would be enough to prove the equality $[K_n:K_{n-1}]=2^{2^{n-1}}$. Using Eq.~\ref{eq:Disc(p_n)} inductively, the first three of the conditions in Eq.~\ref{eq:odd valuation p} imply that $v_p(\text{Disc}(p_{n-1}))=0$, so $p$ does not ramify in $K_{n-1}$. Hence, there is a prime $\mathfrak{q}$ in the ring of integers of $K_{n-1}$ with $v_{\mathfrak{q}}(p) = 1$. It follows that $v_{\mathfrak{q}}(f_nf_{n+1}) = 1$, whence $f_nf_{n+1}$ cannot be a square in $K_{n-1}$. Thus, using Lemma~\ref{lemma:ell(p_n)}, we conclude that
\[ f_nf_{n+1}= \ell(p_{n-1})p_{n-1}(1) = \ell(p_{n-1})p_{n-1}(\phi(\infty))\]
is not a square in $K_{n-1}$. 
We then apply Theorem~\ref{thm:galrat-maximality}, noting that $\gamma_1 = \infty$, $\gamma_2 = 0$, and $\phi(\gamma_2) = \infty$. This gives $[K_n:K_{n-1}]=2^{2^{n-1}}$. Hence, it remains to show the existence of a prime $p$ that satisfies the conditions in Eq.~\ref{eq:odd valuation p}.\par

By Lemma \ref{power of a}, $(f_n)_{n\geq 1}$ is a rigid divisibility sequence. This and the  definition of $\theta_{n+1}$ as the ``primitive part" of $f_{n+1}$ ensure that $\gcd(\theta_{n+1}, f_i) = 1$ for any $i\leq n$. We recall a proof: let $q \in \Z$ be a prime divisor of $\theta_{n+1}$, let $m$ be the minimal positive integer with $q \mid f_m$, let $e = v_q(f_m)$, and let $m' = (n+1)/m$. From condition (2) of Definition \ref{def:rigid} we have $m \mid (n+1)$, whence $m'$ is a positive integer dividing $n+1$. By rigid divisibility we get
\begin{equation} \label{eq:valuation}
v_q(\theta_{n+1}) = e\sum_{rm \mid (n+1)}  \mu\left(\frac{n+1}{rm} \right) = e \sum_{r \mid m'} \mu(r) = \begin{cases}
    e \quad \text{if $m' = 1$} \\
    0 \quad \text{if $m' > 1$}
\end{cases}
\end{equation}
Therefore $v_q(\theta_{n+1}) = 0$ unless $m = n+1$, showing that $\gcd(\theta_{n+1}, f_i) = 1$ for any $i\leq n$.
Since $|\theta_{n+1}|$ is not a square by assumption, there exists a prime $p$ such that $v_p(\theta_{n+1})$ is odd, and $v_p(\theta_{n+1}) =v_p(f_{n+1})$ by Eq. \ref{eq:valuation}. Note that $\gcd(p,a) = 1$ by condition (2) of Lemma \ref{power of a}. Because $a$ is even by hypothesis, this also gives $\gcd(p,2) = 1$. Thus $p$ satisfies all four conditions in Eq.~\ref{eq:odd valuation p}. 
\end{proof}
The following lemma describes, for certain values of $a$, the multiplicative relation between $\beta_n$ and $\theta_n$ in the quotient $\Q^{\times}/\Q^{\times2}$.
\begin{lemma}
\label{lemma:beta_n iff theta_n}
Assume that $a \leq -3$.
Then the following hold.
\begin{enumerate}
\item If $n \geq 3$ is square-free, then $|\theta_n|$ is a square in $\mathbb{Q}$ if and only if $-a\beta_n$ is a square in $\mathbb{Q}$.
\item If $n \geq 3$ is not square-free, then $|\theta_n|$ is a square in $\mathbb{Q}$ if and only if $\beta_n$ is a square in $\mathbb{Q}$.
\end{enumerate}
\begin{proof}
By definitions and using Lemma~\ref{power of a}, we have
\begin{align*}
\beta_n = \prod_{d|n}(p_d(0))^{\mu(n/d)}
= \prod_{d|n}(a^{2^{d-1}}f_d)^{\mu(n/d)}
= \bigg(\prod_{d|n}a^{2^{d-1}\mu(n/d)}\bigg)\theta_n.
\end{align*}
Set $C_n = \prod_{d|n}a^{2^{d-1}\mu(n/d)}$, and note that $C_n$ is $a^{\mu(n)}$ times a square in $\Q$.  If $n$ is square-free, we thus have $\overline{C_n} = \overline{a}$, and since $\beta_n>0$ by Lemma~\ref{lemma:sign of p_n} (here we use $n \geq 3)$, we conclude $\theta_n<0$. It follows that $|\theta_n|$ is a square in $\mathbb{Q}$ if and only if $-a\beta_n$ is a square in $\mathbb{Q}$, proving (1). If $n$ is not square-free, then $C_n$ is a square in $\mathbb{Q}$, implying $\theta_n>0$ (since $\beta_n>0$ again by Lemma~\ref{lemma:sign of p_n}). It follows that $|\theta_n|$ is a square in $\mathbb{Q}$ if and only if $\beta_n$ is a square in $\mathbb{Q}$, proving (2). 
\end{proof}
\end{lemma}

The following lemma is a key tool in the proof of Theorem~\ref{thm:surjectivityA}. We obtain it by adapting arguments of Stoll \cite[Lemmas 2.1 and 2.2]{stoll}.  Recall that the radical of an integer $n$ is product of its distinct prime divisors, written $\text{rad}(n)$.
\begin{lemma}
\label{lemma:rad-divisibility}
Let $\phi = p(z)/q(z)\in \Q(z)$ have degree $d\geq 2$, where $p,q \in \Z[z]$ are relatively prime. Assume that $p$ and $q$ are even polynomials and $q$ is the square of a polynomial in $\Z[z]$. Let $p_n, q_n$ be as in Eq. \ref{eq:p_n,q_n}, and let 
$$\beta_{\alpha, n} = \prod_{d \mid n}(p_d(\alpha))^{\mu(n/d)}.$$
Let $n\geq 2$ be a positive integer and let   $k=\frac{n}{\text{rad}(n)}$. 
Assume that $\alpha\in \mathbb{Q}$ satisfies $\phi^i(\alpha)\neq \infty$ for all $i\geq k$ and $\phi^i(\alpha)\neq 0$ for all $i\geq 1$. Suppose that there exists a positive integer $m$ such that:
\begin{enumerate}
    \item $v_{\ell}(\phi^k(\alpha)) = 0$ for every prime divisor $\ell$ of $m$.
    \item $\phi^k(\alpha) \equiv - \phi^{k+1}(\alpha) \pmod{m}$,
    \item $-1$ is not a square modulo $m$,

\end{enumerate}
Then $\beta_{\alpha, n}$
is not a square in $\mathbb{Q}$.
\end{lemma}
\begin{proof}
Let $\delta=\phi^{k+1}(\alpha)$, and pick a positive integer $m$ that satisfies the given conditions. Observe that $\delta \in (\Z/m\Z)^\times$ by conditions (1) and (2). We will induct on $i$ to show that $\phi^{k+i}(\alpha)\equiv \delta \pmod{m}$ for all $i\geq 1$. We know the base case $i=1$. Suppose that the congruence holds for $i=N$. Then 
\[\phi^{N+1}(\alpha)\equiv\phi(\phi^N(\alpha)))\equiv \phi(\delta)\equiv \phi(-\delta)\equiv \phi (\phi^k(\alpha))\equiv \phi^{k+1}(\alpha)\equiv \delta \pmod{m},\]
which completes the induction. Note that in the third congruence we used that $\phi$ is an even function, and in the fourth congruence we used that $\phi^k(\alpha)\equiv -\delta \pmod{m}$ from condition (1). Taking $i$ to be multiples of $k$ then gives
\begin{equation}
\label{eq:ell>2}
\phi^{kj}(\alpha)\equiv \delta \pmod{m} \quad \text{ for }j \geq 2.
\end{equation}
We now write
\[\beta_{\alpha, n} = \prod_{d \mid n} (p_d(\alpha))^{\mu(n/d)} = \prod_{t \mid \text{rad}(n)}(p_{kt}(\alpha))^{\mu(\text{rad}(n)/t)}.\]
Note that $\beta_{\alpha, n}$ is well-defined because $p_i(\alpha) \neq 0$ for all $i \geq 1$ by hypothesis. 
By assumption $q$ is the square of a polynomial, whence we have $\overline{p_{kt}(\alpha)} = \overline{\phi^{kt}(\alpha)}$. Therefore it suffices to show that 
\[\gamma_n:=\prod_{t|\text{rad}(n)}(\phi^{kt}(\alpha))^{\mu(\text{rad}(n)/t)}\]
is not a square in $\mathbb{Q}$. Recalling $\phi^k(\alpha)\equiv -\delta \pmod{m}$ and using Eq.~\ref{eq:ell>2}, we get
\[\gamma_n \equiv (-1)^{\mu(\text{rad}(n))}\prod_{t|\text{rad}(n)} \delta^{\mu(\text{rad}(n)/t)}\equiv -1 (\text{mod }m),\]
where we used the facts that $\mu(\text{rad}(n))=\pm 1$ and $\sum_{t|\text{rad}(n)}\mu(\text{rad}(n)/t)=0$ in the second congruence. Since $-1$ is not a square modulo $m$ by assumption, $\gamma_n$ cannot be a square in $\mathbb{Q}$, as desired.
\end{proof}

\begin{lemma}
\label{stability{a=2(mod 4)}}
Let $\phi(z) (z^2 + a)/z^2$ for some integer $a \equiv 2 \pmod{4}$, and let $p_n, q_n$ be as in Eq. \ref{eq:p_n,q_n mainfamily}. Then $p_n$ is irreducible over $\mathbb{Q}$ for all $n\geq 1$.
\end{lemma}
\begin{proof}
An induction using Eq. \ref{fngn} shows
\begin{equation} \label{equivalence}
\text{$f_n \equiv 1 + a \pmod 8$ for all $n \geq 3$.}
\end{equation}
 From condition (2) of Lemma \ref{lemma:ell(p_n)} we then have, for $n \geq 1$,
$$
\text{$p_n(1) = f_{n+2} \equiv 1 + a \equiv 3 \pmod{4}$}.
$$
Hence $p_n(1)$ is not a square in $\Q$ for any $n \geq 1$. Moreover, $p_1 = z^2+a$ is irreducible over $\Q$ because $-a \equiv 2 \pmod{4}$. An application of Lemma \ref{lemma:stability} finishes the proof. 
\end{proof}

We are now ready to prove the main results of this section. 

\begin{proof}[Proof of Theorem \ref{thm:surjectivityA}]
Because $n$ is not squarefree, we have $n \geq 4$, and so by Lemmas \ref{lemma:sufficient-theta_n}, \ref{lemma:beta_n iff theta_n}, and \ref{stability{a=2(mod 4)}} it suffices to show that $\beta_n$ is not a square in $\Q$. To prove this, we appeal to Lemma \ref{lemma:rad-divisibility} with $\alpha = 0$. Because $a \equiv 2 \pmod{4}$, $-a$ is not a square in $\Q$, and thus $\phi^{-1}(0) \cap \PP^1(\Q) = \emptyset$. Hence $\phi^i(0) \neq 0$ for all $i \geq 1$. Because $\phi^{-1}(\infty) = \{0\}$, this also gives $\phi^i(0) \neq \infty$ for all $i \geq 2$. Because $n$ is assumed not square-free, we have that $k = n/\text{rad}(n)\geq 2$, and so this implies $\phi^i(0) \neq \infty$ for all $i \geq k$.

If $k = 2$, observe that $\phi^k(0) =1$ and $\phi^{k+1}(0) = a+1$. Because $a \equiv 2 \pmod{4}$ it follows that $\phi^k(0) \equiv -\phi^{k+1}(0) \pmod{4}$. Taking $m = 4$ in Lemma~\ref{lemma:rad-divisibility} then gives that $\beta_n$ is not a square in $\mathbb{Q}$.\par

Suppose for the remainder of the proof that $k>2$. To be able to apply Lemma~\ref{lemma:rad-divisibility}, we will study the integer $\phi^k(0)+\phi^{k+1}(0)$. Note that for any $n$ we have $\phi^n(0) = p_n(0)/q_n(0) = p_n(0)/(p_{n-1}(0))^2$. From condition (1) of Lemma \ref{power of a} we then have 
\[\phi^k(0)+\phi^{k+1}(0) 
= \frac{p_k(0)}{(p_{k-1}(0))^2}+\frac{p_{k+1}(0)}{(p_{k}(0))^2} = \frac{f_k}{f_{k-1}^2}+\frac{f_{k+1}}{f_{k}^2} = \frac{f_k^3 +f_{k+1}f_{k-1}^2}{f_k^2f_{k-1}^2}.\]
Set $A_k=f_k^3 +f_{k+1}f_{k-1}^2$ and $B_k=f_k^2f_{k-1}^2$. We start by showing that $\gcd(A_k, B_k)=1$. By Lemma \ref{power of a}, $(f_n)_{n \geq 1}$ is a rigid divisibility sequence with $f_1 = 1$, and it follows that  
\begin{equation} \label{gcd}
\text{$\gcd(f_n, f_{n-1}) = 1$ for all $n \geq 2$.} 
\end{equation}
If $\ell$ is a prime dividing $B_k$, then $\ell \mid f_k$ or $\ell \mid f_{k-1}$. In the former case, $\ell$ cannot also divide $A_k$, for then $\ell \mid f_{k+1}f_{k-1}^2$, contradicting Eq. \ref{gcd}. In the latter case, again $\ell$ cannot divide $A_k$, for otherwise $\ell \mid f_k^3$, producing a contradiction to Eq. \ref{gcd}. Therefore $\gcd(A_k, B_k) = 1$.

Observe that $a \equiv 2 \pmod{4}$ implies that each $f_n$ is odd, and thus $f_n^2 \equiv 1 \pmod{8}$. We apply Eq. \ref{equivalence} (possible since $k > 2$) to obtain
\begin{equation}
\label{eq:A_k}
A_k \equiv f_k + f_{k+1} \equiv 2 + 2a \equiv 6 \pmod{8}.
\end{equation}
Therefore $A_k$ has a prime divisor $p$ such that $p \equiv 3 \pmod{4}$. Since $\gcd(A_k, B_k) = 1$, we obtain $v_p(\phi^k(0) + \phi^{k+1}(0)) > 0$, and thus $\phi^k(0) \equiv -\phi^{k+1}(0) \pmod{p}$. This verifies conditions (2) and (3) of Lemma~\ref{lemma:rad-divisibility} for $\alpha=0$ and $m=p$. To complete the application of Lemma \ref{lemma:rad-divisibility} it only remains to verify $v_p(\phi^k(0)) = 0$. But $\gcd(A_k, B_k) = 1$ implies $p \nmid B_k = f_k^2f_{k-1}^2$, whence $p$ divides neither $f_k$ nor $f_{k-1}$, and so $v_p(\phi^k(0)) = v_p(f_k/f_{k-1}^2) = 0$.
\end{proof}

\begin{proof}[Proof of Theorem~\ref{thm:surjectivityB}]
It suffices to prove that $[K_n : K_{n-1}] = 2^{2^{n-1}}$ for all $n \geq 1$. 
Observe that $-2(2m^2 - 1)^2 \equiv -2 \cdot 1 \equiv 2 \pmod{4}$, so we may apply Theorem \ref{thm:surjectivityA} to conclude $[K_n : K_{n-1}] = 2^{2^{n-1}}$ if $n$ is not square-free.

Assume that $n$ is square-free. i.e. $k = n/\text{rad}(n) = 1$. If $n = 1$, then Lemma~\ref{stability{a=2(mod 4)}} gives $[K_1: K] = 2$, which dispenses with this case. If $n \geq 2$ then Lemma~\ref{stability{a=2(mod 4)}} gives that $p_{n-1}$ is irreducible over $\Q$, and thus by Lemma \ref{lemma:sufficient-theta_n} we need only show that $|\theta_{n+1}|$ is not a square in $\Q$. We will do this by showing that $|\theta_i|$ is not a square in $\Q$ for any $i \geq 3$. 

We encounter two obstacles in using Lemma~\ref{lemma:rad-divisibility} directly to conclude that $|\theta_i|$ is not a square. The first is that from Lemma \ref{lemma:beta_n iff theta_n} we must show that $-a\beta_i$ is not a square in $\Q$, where $a = -2(2m^2 -1)^2$. The second is that in this case $k = 1$ and thus $\phi^k(0)=\phi(0) = \infty$. We will circumvent these by modifying the argument in Lemma \ref{lemma:beta_n iff theta_n}, which took $\alpha = 0$, to work with different $\alpha$. In particular, we get around the fact that $\phi(0) = \infty$ by taking $a$ and $\alpha$ such that $\alpha \neq 0$ and $\phi^3(\alpha) = \phi^3(0)$. This implies $\phi^i(\alpha) = \phi^i(0)$ for all $i \geq 3$, which makes the ratio of $\beta_i$ and $\beta_{\alpha, i}$ relatively nice in the quotient $\Q^\times/\Q^{\times 2}$. We find such $a$ and $\alpha$ by considering the curve defined by
\begin{equation} \label{curve}
\phi_t^3(z) = \phi_t^3(0),
\end{equation}
where $\phi_t(z) = (z^2 + t)/z^2$. This curve has two components, one with $z = 0$ and one with $z \neq 0$. Fortunately, the latter component has genus zero, and one calculates the rational parametrization 
$$
t = -2(2m^2 - 1)^2 \qquad z = \frac{2m^2 - 1}{m}.
$$
We note that the curve in Eq. \ref{curve} uses the third iterate of $\phi_t$ because smaller iterates do not have a component with $z \neq 0$.
We thus take $a = -2(2m^2 - 1)^2$ and $\alpha = \frac{2m^2-1}{m}$, and note that
\[\phi(\alpha) = 1-2m^2,\text{ }\phi^2(\alpha) = -1,\text{ and }\phi^3(\alpha) = \phi^3(0) = -2(2m^2-1)^2+1.\]
In particular, we have 
\begin{equation}
\label{eq:alpha vs 0}
\phi^i(\alpha) = \phi^i(0)\text{ for any }i\geq 3.
\end{equation}
Note that $p_1(z) = \phi(z)z^2$ and $p_i(z) = \phi^i(z)(p_{i-1}(z))^2$ for $i \geq 2$. Therefore
\begin{equation} 
\label{eq:p_i(alpha) vs phi^i(alpha)}
\text{$\overline{p_i(\alpha)} = \overline{\phi^i(\alpha)}$ for any $i\geq 1$ \text{ and } $\overline{p_i(0)} = \overline{\phi^i(0)}$ for any $i\geq 2$.}
\end{equation}

\textbf{Case 1.} 
Suppose that $n$ is odd.
Recalling the definitions of 
$\beta_{\alpha,n}$ and $\beta_n$ from p. \pageref{betadef}, Eq.~\ref{eq:alpha vs 0} and Eq. \ref{eq:p_i(alpha) vs phi^i(alpha)} yield

\[\overline{\beta_{\alpha,n}} = \overline{\beta_n\bigg(\frac{p_1(\alpha)}{p_1(0)}\bigg)}.\]
We then have 
\begin{align}
\begin{split}
\label{eq:beta_n up to square}
\overline{\beta_{\alpha, n}} = \overline{\beta_n\bigg(\frac{p_1(\alpha)}{p_1(0)}\bigg)}
= \overline{\beta_n\bigg(\frac{\phi(\alpha)}{p_1(0)}\bigg)}
= \overline{\beta_n\bigg(\frac{1-2m^2}{-2(2m^2-1)^2}\bigg)}
= \overline{\beta_n(4m^2-2)}.
\end{split}
\end{align}
From Lemma \ref{lemma:beta_n iff theta_n} and the fact that we have taken $a = -2(2m^2-1)^2$, we have $\overline{\beta_n} = \overline{2\theta_n}$, whence from Eq. \ref{eq:beta_n up to square} and Eq. \ref{eq:p_i(alpha) vs phi^i(alpha)} we have
\begin{equation}
\label{goalthing}
\overline{|\theta_n|} = \overline{\beta_{\alpha,n}(2m^2 - 1)} = \overline{(2m^2-1)\prod_{d|n}(\phi^d(\alpha))^{\mu(n/d)}}.
\end{equation}
 
By hypothesis $p$ is a prime dividing either $2m-1$ or $2m+1$ and with $p\equiv 5$ or $7\text{ }(\text{mod }8)$. Note that 
\begin{align*}
\phi^3(\alpha)-\phi(\alpha) &= (-2(2m^2-1)^2+1)-(1-2m^2)\\
&= -2(2m^2-1-m)(2m^2-1+m)\\
&= -2(2m+1)(m-1)(2m-1)(m+1),
\end{align*}
and thus $\phi^3(\alpha) \equiv \phi(\alpha) \pmod{p}$. It follows that 
$$\phi^{2j + 1}(\alpha) \equiv \phi(\alpha) \equiv 1 - 2m^2 \equiv 1 - (1/2)(2m)^2 \equiv 1/2 \pmod{p}$$ 
for all $j \geq 1$. Using Eq. \ref{eq:alpha vs 0} we then obtain $\phi^n(0) \not\equiv 0 \pmod{p}$ for all $n \geq 1$, from which it follows that $|\theta_n| \not\equiv 0 \pmod{p}$ for all $n \geq 1$.

Because $n$ is odd, so are all divisors of $n$, so we get
\begin{equation}
\label{eq:mobius=1 for 1/2}
\prod_{d|n} (\phi^d(\alpha))^{\mu(n/d)}\equiv \prod_{d|n} \left(\frac{1}{2}\right)^{\mu(n/d)} \equiv 1\text{ }(\text{mod }p)
\end{equation}
since $\sum_{d|n} \mu(n/d) = 0$ (recall our earlier assumption $n \geq 2$). 
From \ref{goalthing} we have that 
\[
|\theta_n| = (2m^2-1)\bigg(\prod_{d|n}(\phi^d(\alpha))^{\mu(n/d)}\bigg)X_n^2
\]
for some $X_n\in \mathbb{Q}^{\times}$.
We have $2m^2-1\equiv \frac{1}{2}-1\equiv -\frac{1}{2} \pmod{p}$, and from Eq.~\ref{eq:mobius=1 for 1/2} we conclude that 
\[
|\theta_n|\equiv -\frac{1}{2}X_n^2 \pmod{p},
\] 
where $X_n \not\equiv 0 \pmod{p}$ because $|\theta_n| \not\equiv 0 \pmod{p}$.
But $p\equiv 5$ or $7\text{ }(\text{mod }8)$ ensures that $-2$ is not a square modulo $p$, 
whence $|\theta_n|$ cannot be a square in $\mathbb{Q}$.\par

\textbf{Case 2.} Suppose that $n$ is even. Eq.~\ref{eq:alpha vs 0} and Eq. \ref{eq:p_i(alpha) vs phi^i(alpha)} yield
\[
\overline{\beta_{\alpha, n}} = \overline{\beta_n\bigg(\frac{p_1(\alpha)p_2(\alpha)}{p_1(0)p_2(0)}\bigg)}.
\]
Note that $p_1(z) = z^2 + a$ and $p_2(z) = p_1(z)^2 + aq_1(z)^2 = (1 + a)z^4 + 2az^2 + a^2$. Therefore
\begin{align}
\begin{split}
\label{eq:beta_n up to square-Case 2}
\overline{\beta_{\alpha, n}} = \overline{\beta_na} \overline{\bigg(\frac{p_2(\alpha)}{p_1(\alpha)}\bigg)}
= \overline{\beta_na} \overline{\bigg(\frac{\phi^2(\alpha)}{\phi(\alpha)}\bigg)}
= \overline{\beta_na}\overline{\bigg(\frac{-1}{1-2m^2}\bigg)}
= \overline{\beta_n} \overline{\bigg(\frac{2}{1-2m^2}\bigg)}.
\end{split}
\end{align}
As in case 1, Lemma \ref{lemma:beta_n iff theta_n},  Eq. \ref{eq:beta_n up to square-Case 2}, and Eq. \ref{eq:p_i(alpha) vs phi^i(alpha)} give
\begin{equation*} 
\overline{|\theta_n|} = \overline{(1-2m^2)\beta_{\alpha, n}} = \overline{(1-2m^2)\prod_{d|n}(\phi^d(\alpha))^{\mu(n/d)}}.
\end{equation*}
Thus there exists $X_n\in \mathbb{Q}^{\times}$ such that
\begin{equation} \label{thetalatest}
   |\theta_n| = (1-2m^2)\bigg(\prod_{d|n}(\phi^d(\alpha))^{\mu(n/d)}\bigg)X_n^2. 
\end{equation}

By hypothesis $p$ is a prime dividing one of $m, m-1$, or $m + 1$, and with $p \equiv 3 \pmod{4}$.

\textbf{Case 2a.} If $p \mid (m-1)$ or $p \mid (m+1)$, note that 
$$\phi^2(\alpha) = \phi(\alpha) = -1 - (1-2m^2) = 2(m^2 - 1),$$
and so $\phi^2(\alpha) \equiv \phi(\alpha) \pmod{p}$. It follows that $\phi^j(\alpha) \equiv \phi(\alpha) \equiv 1 - 2m^2 \equiv -1 \pmod{p}$ for all $j \geq 2$, which by Eq. \ref{eq:alpha vs 0} gives $|\theta_n| \not\equiv 0 \pmod{p}$ for all $n \geq 1$. We now have 
\begin{equation}
\label{eq:mobius=1}
\prod_{d|n} (\phi^d(\alpha))^{\mu(n/d)}\equiv \prod_{d|n} (-1)^{\mu(n/d)} \equiv 1\text{ }(\text{mod }p)
\end{equation}
since $\sum_{d|n} \mu(n/d) = 0$ (recall our earlier assumption $n\geq2$.)
We have $1 - 2m^2 \equiv -1 \pmod{p}$, and Eq. \ref{thetalatest} and Eq. \ref{eq:mobius=1} then give
$$
|\theta_n| \equiv -X_n^2 \pmod{p}.
$$
where $X_n \not\equiv 0 \pmod{p}$ because $|\theta_n| \not\equiv 0 \pmod{p}$. But $p \equiv 3 \pmod{4}$, so $-1$ is not a square modulo $p$, proving that $|\theta_n|$ is not a square in $\Q$.

\textbf{Case 2b.} If $p \mid m$, note that
$\phi^2({\alpha})+\phi(\alpha) = -1+(1-2m^2) = -2m^2$, and thus $\phi^2(\alpha) \equiv -\phi(\alpha) \pmod{p}$. Because $\phi$ is an even function, we have
$$
\phi^j(\alpha) \equiv \phi^{j-2}(\phi^2(\alpha)) \equiv \phi^{j-2}(-\phi(\alpha)) \equiv \phi^{j-2}(\phi(\alpha)) \equiv \phi^{j-1}(\alpha) \pmod{p}
$$
for all $j \geq 3$. It follows that $\phi(\alpha) \equiv 1 \pmod{p}$ and $\phi^j(\alpha) \equiv -1 \pmod{p}$ for all $j \geq 2.$ Recalling that $n \geq 2$, we then obtain 
\begin{equation}
\label{eq:mobius=1-second solution}
\prod_{d|n} (\phi^d(\alpha))^{\mu(n/d)}\equiv - \prod_{d|n} (-1)^{\mu(n/d)}  \equiv -1 \pmod{p}.
\end{equation}
We have $1-2m^2\equiv 1\pmod{p}$, and Eq. \ref{thetalatest} and Eq. \ref{eq:mobius=1-second solution} then give
$$
|\theta_n| \equiv -X_n^2 \pmod{p}.
$$
But $p \equiv 3 \pmod{4}$, so $-1$ is not a square modulo $p$, proving that $|\theta_n|$ is not a square in $\Q$.
\end{proof}

\bibliographystyle{plain}

\end{document}